\documentclass[12pt,oneside]{amsart}

\usepackage[english]{babel}
\usepackage{amsmath, amsthm, amsfonts, amssymb}
\usepackage{mathrsfs}
\usepackage[pdftex]{graphicx}
\usepackage{hyperref}
\usepackage{calc}
\usepackage{caption}
\usepackage{subcaption}
\usepackage{multicol}
\usepackage{color}

\numberwithin{figure}{section}
\numberwithin{table}{section}

\renewcommand\dim[1]{\operatorname{dim}(#1)}

\newcommand\Nat{\mathbb{N}}

\newcommand\Ganz{\mathbb{Z}}

\newcommand\Komplex{\mathbb{C}}

\newcommand\Betrag[1]{|#1|}

\newcommand\PolyRing[2]{#1[#2]}

\newcommand\FFSA{free filtration subarrangement }
\newcommand\FFSAp{free filtration subarrangement}
\newcommand\FFSAs{free filtration subarrangements }
\newcommand\FFSAsp{free filtration subarrangements}
\newcommand\FF{free filtration }

\newcommand\GeneralLinearGroup[1]{\operatorname{GL}(#1)}

\newcommand\Abb[3]{#1:#2 \to #3}
\newcommand\Derivation[1]{\operatorname{Der}(#1)}

\newcommand\Kern[1]{\operatorname{ker}(#1)}

\newcommand\IFC{\mathcal{IF}}
\newcommand\RFC{\mathcal{RF}}
\newcommand\DFC{\mathcal{DF}}
\newcommand\AnV[2]{(\mathcal{\uppercase{#1}},#2)}
\newcommand\An[1]{\mathcal{\uppercase{#1}}}
\newcommand\A{\mathcal{A}}
\newcommand\B{\mathcal{B}}
\newcommand\C{\mathcal{C}}

\newcommand\N{\mathcal{N}}
\newcommand\F{\mathcal{F}}
\newcommand\crg[1]{G_{#1}}
\newcommand\crgM[3]{G(#1,#2,#3)}
\newcommand\Arr[1]{\mathcal{A}(#1)}
\newcommand\expA[2]{\operatorname{exp}(#1) = \{\{ #2 \}\}}
\newcommand\expAA[1]{\operatorname{exp}(#1)}
\newcommand\LA[1]{L(#1)}
\newcommand\LAq[2]{L_{#2}(#1)}
\newcommand\emptA[1]{\Phi_{#1}}
\newcommand\DerA[1]{D(#1)}
\newcommand\ddxi[1]{\partial/\partial x_{#1}}
\newcommand\pdeg[1]{\operatorname{pdeg}#1}
\newcommand\CharPolyA[1]{\chi(#1,t)}

\newcommand\TripleA[1]{(#1,#1',#1'')}
\newcommand\SymA[1]{S(#1^*)}
\newcommand\PolyA[1]{Q(#1)}
\newcommand\SAA{\mathscr{A}}

\numberwithin{equation}{section}

\theoremstyle{plain}
\newtheorem{lemma}[equation]{Lemma}
\newtheorem{theorem}[equation]{Theorem}

\newtheorem{corollary}[equation]{Corollary}
\newtheorem{proposition}[equation]{Proposition}

\theoremstyle{definition}
\newtheorem{definition}[equation]{Definition}
\newtheorem{remark}[equation]{Remark}
\newtheorem{example}[equation]{Example}

\begin{document}

\title{Recursively free reflection arrangements}
\author{Paul M\"ucksch}
\address%
{Institut f\"ur Algebra, Zahlentheorie und Diskrete Mathematik,
Fakult\"at f\"ur Mathematik und Physik,
Leibniz Universit\"at Hannover,
Welfengarten 1, D-30167 Hannover, Germany}
\email{muecksch@math.uni-hannover.de}

%\date{\today} 

\maketitle

\begin{abstract}
Let $\A = \A(W)$ be the reflection arrangement of the finite complex reflection group $W$.
By Terao's famous theorem, the arrangement $\A$ is free. 
In this paper we classify all reflection arrangements which belong to the smaller class of recursively free arrangements.
Moreover for the case that $W$ admits an irreducible factor isomorphic to $\crg{31}$ 
we obtain a new (computer-free) proof for the non-inductive freeness of $\A(W)$.
Since our classification implies the non-recursive freeness of the reflection arrangement $\A(\crg{31})$,
we can prove a conjecture by Abe about the new class of divisionally free arrangements which he recently introduced.
\end{abstract}

\section{Introduction}
Suppose that $W$ is a finite complex reflection group acting on the complex vector space $V$.
Let $\A = \A(W)$ be the associated hyperplane arrangement of the reflecting hyperplanes of $W$.
Then $\A$ is free, \cite{Terao1980}. 

There are several stronger notions of freeness.
In this paper we are mainly interested in two, namely \emph{inductive freeness}, first introduced by Terao in \cite{Terao1980}
and \emph{recursive freeness} which was introduced by Ziegler in \cite{Zielger87}.

In \cite{MR2854188} Barakat and Cuntz proved that all Coxeter arrangements are inductively free
and in \cite{2012arXiv1208.3131H} Hoge and R\"ohrle completed the classification of inductively free reflection arrangements
by inspecting the remaining complex cases.
They gave an easy characterization for all the cases but one, namely if the complex reflection group $W$ admits 
an irreducible factor isomorphic to $\crg{31}$ and handling this case also turns out to be the most difficult part of this paper.

In \cite{MR3272729} Cuntz and Hoge gave first examples for free but not recursively free arrangements.
One of them is the reflection arrangement of the exceptional complex reflection group $\crg{27}$.

Very recently, Abe, Cuntz, Kawanoue, and Nozawa \cite{2014arXiv1411.3351A} 
found smaller examples (with $13$ hyperplanes, being the smallest possible, and with $15$ hyperplanes)
for free but not recursively free arrangements in characteristic $0$.

Nevertheless, free but not recursively free arrangements seem to be rare.

Since reflection arrangements play an important role in the theory of hyperplane arrangements, it is natural to ask
which other reflection arrangements are free but not recursively free.   
In this paper we answer this question and complete the picture for reflection arrangements
by showing which of the not inductively free reflection arrangements are recursively free
and which are free but not recursively free. 
We obtain a classification of all recursively free reflection arrangements:

\begin{theorem}\label{thm:RFcrArr}
For $W$ a finite complex reflection group, the reflection arrangement $\A(W)$ of $W$ is
recursively free if and only if $W$ does not admit an irreducible factor isomorphic to one of the 
exceptional reflection groups $\crg{27}, \crg{29}, \crg{31}, \crg{33}$ and $\crg{34}$.
\end{theorem}

Furthermore, for the special case $W \cong \crg{31}$, we obtain a (with respect to ``Addition'' and ``Deletion'') isolated cluster 
of free but not recursively free subarrangements of $\A(W)$ in dimension $4$.

Recently in \cite{AbeDivFree2015}, Abe introduced the new class of divisionally free arrangements, based
on his Division-Theorem, \cite[Thm.~1.1]{AbeDivFree2015}, about freeness 
and division of characteristic polynomials analogous to the class of 
inductively free arrangements based on the Addition-Deletion-Theorem.
With Theorem \ref{thm:RFcrArr}, we are able to positively settle a conjecture by Abe, \cite[Conj.~5.11]{AbeDivFree2015}, 
about this new class of free arrangements, which we state as the next theorem.

\begin{theorem}\label{thm:ConjAbe}
There is an arrangement $\A$ such that $\A \in \DFC$ and $\A \notin \RFC$.
\end{theorem}

Finally, we will comment on the situation of a restriction of a reflection arrangement.

In order to compute the different intersection lattices of the reflection arrangements in question,
to obtain the respective invariants, and to recheck our results
we used the functionality of the GAP computer algebra system, \cite{GAP4}.

The author thanks his thesis advisor Professor Michael Cuntz and Torsten Hoge for many helpful discussions and remarks.

\section{Recollection and Preliminaries}

We review the required notions and definitions. Compare with \cite{orlik1992arrangements}.

\subsection{Arrangements of hyperplanes}

\begin{definition}
An $\ell$\emph{-arrangement of hyperplanes} is a pair $\AnV{A}{V}$, where $\A$ is a finite collection of hyperplanes
(codimension $1$ subspaces) in $V = \mathbb{K}^\ell$, a finite dimensional vector space over a fixed field $\mathbb{K}$.
For $\AnV{A}{V}$ we simply write $\A$ if the vector space $V$ is unambiguous.

We denote the empty $\ell$-arrangement by $\emptA{\ell}$.
\end{definition}

In this paper we are only interested in complex \emph{central} arrangements $\A$, that is, all the hyperplanes in $\A$ 
are linear subspaces and $V$ is a finite dimensional complex vector space $V = \Komplex^\ell$.

If we want to explicitly write down the hyperplanes of an $\ell$-arrangement,
we will use the notation: 
$H = \Kern{\alpha} =: \alpha^\perp$ for a linear form $\alpha \in V^*$.
If $\{x_1,\ldots,x_n\}$ is a basis for $V^*$, we write $\alpha = \sum_{i=1}^\ell a_i x_i \in V^*$ 
also as a row vector $(a_1,\ldots,a_\ell)$.

The \emph{intersection lattice} $\LA{\A}$ of $\A$ is the set of all subspaces $X$ of $V$ of the form
$X = H_1 \cap \ldots \cap H_r$ with $\{ H_1,\ldots,H_r\} \subseteq \A$.
If $X \in \LA{\A}$, then the \emph{rank} $r(X)$ of $X$ is defined as $r(X) := \ell - \dim{X}$ and the rank of the arrangement
$\A$ is defined as $r(\A) := r(T(\A))$ where $T(\A) := \bigcap_{H \in \A} H$ is the \emph{center} of $\A$.
An $\ell$-arrangement $\A$ is called \emph{essential} if $r(\A)=\ell$.
For $X \in \LA{\A}$, we define the localization
$\A_X := \{ H \in \A \mid X \subseteq H \}$ of $\A$ at $X$, and the \emph{restriction of} $\A$ to $X$,
$(\A^X,X)$, where $\A^X := \{ X\cap H \mid H \in \A \setminus \A_X \}$.
For $0 \leq q \leq \ell$ we write $\LAq{\A}{q} := \{ X \in \LA{\A} \mid r(X) = q \}$.
For $\A \neq \emptA{\ell}$, let $H_0 \in \A$.
Define $\A' := \A \setminus \{ H_0 \}$, and $\A'' := \A^{H_0}$.
We say that $\TripleA{\A}$ is a \emph{triple} of arrangements (with respect to $H_0$), \cite[Def.~1.14]{orlik1992arrangements}.

The \emph{product} $\A = (\A_1 \times \A_2,V_1 \oplus V_2)$ of two arrangements $(\A_1,V_1)$, $(\A_2,V_2)$
is defined by
\begin{equation*}
\A := \A_1 \times \A_2 = \{ H_1 \oplus V_2 \mid H_1 \in \A_1 \} \cup \{ V_1 \oplus H_2 \mid H_2 \in \A_2 \},
\end{equation*}
see \cite[Def.~2.13]{orlik1992arrangements}. In particular $\Betrag{\A} = \Betrag{\A_1} + \Betrag{\A_2}$.

If an arrangement $\A$ can be written as a non-trivial product $\A = \A_1 \times \A_2$,
then $\A$ is called \emph{reducible}, otherwise \emph{irreducible}.
%
%Let $\A = \A_1 \times \A_2$ be a product. By \cite[Prop.\ 2.4]{orlik1992arrangements}, there is a lattice isomorphism
%\begin{equation*}
%\LA{\A_1} \times \LA{\A_2} \to \LA{\A} \text{, } (X_1,X_2) \mapsto X_1 \oplus X_2,
%\end{equation*}
%and one can easily see that for $X = X_1 \oplus X_2 \in \LA{\A}$ we have
%\begin{align}
%\A_X = \quad &(\A_1)_{X_1} \times (\A_2)_{X_2} \, , \label{eq_AxA_A_X} \\
%\A^X = \quad &(\A_1)^{X_1} \times (\A_2)^{X_2}. \label{eq_AxA_AX}
%\end{align}

For an arrangement $\A$ the \emph{M\"obius function} $\Abb{\mu}{\LA{\A}}{\Ganz}$ is defined by:
\begin{equation*}
\mu(X) = \left\{ \begin{array}{l l}
   1 & \quad \text{if } X=V \text{,}\\
    -\sum_{V \supseteq Y \supsetneq X} \mu(Y) & \quad \text{if }X \neq V\text{.}
  \end{array} \right.
\end{equation*}

We denote by $\CharPolyA{\A}$ 
the \emph{characteristic polynomial} of $\A$ which is defined by:
\begin{equation*}
	\CharPolyA{\A} = \sum_{X \in \LA{\A}} \mu(X) t^{\dim{X}}.
\end{equation*}

An element $X \in \LA{\A}$ is called \emph{modular} if $X + Y \in \LA{\A}$ for all $Y \in \LA{\A}$.
An arrangement $\A$ with $r(\A) = \ell$ is called \emph{supersolvable} if the intersection lattice $\LA{\A}$ is supersolvable, i.e.\ it has
a maximal chain of modular elements $ V = X_0 \supsetneq X_1 \supsetneq \ldots \supsetneq X_\ell = T(\A)$, 
$X_i \in \LA{\A}$ modular.
For example an essential $3$-arrangement $\A$ is supersolvable if there exists an $X \in \LA{\A}_2$ which is
connected to all other $Y \in \LA{\A}_2$ by a suitable hyperplane $H \in \A$, (i.e.\ $X + Y \in \A$).

\subsection{Free Arrangements}

Let $S = \SymA{V}$ be the symmetric algebra of the dual space $V^*$ of $V$. If $x_1,\ldots,x_\ell$ is a basis of $V^*$,
then we identify $S$ with the polynomial ring $\PolyRing{\Komplex}{x_1,\ldots,x_\ell}$ in $\ell$ variables.
The algebra $S$ has a natural $\Ganz$-grading:
Let $S_p$ denote the $\Komplex$-subspace of $S$ of the homogeneous polynomials of degree $p$ ($p \in \Nat_{\geq 0}$),
then $S = \bigoplus_{p \in \Ganz} S_p$, where $S_p = 0$ for $p < 0$.

Let $\Derivation{S}$ be the $S$-module of $\Komplex$-derivations of $S$. It is a free $S$-module with basis
$D_1,\ldots,D_\ell$ where $D_i$ is the partial derivation $\ddxi{i}$.
We say that $\theta \in \Derivation{S}$ is \emph{homogeneous of polynomial degree} $p$ provided
$\theta = \sum_{i=1}^\ell f_i D_i$, with $f_i \in S_p$ for each $1 \leq i \leq \ell$.
In this case we write $\pdeg{\theta} = p$.
With this definition we get a $\Ganz$-grading for the $S$-module $\Derivation{S}$:
Let $\Derivation{S}_p$ be the $\Komplex$-subspace of $\Derivation{S}$ consisting of 
all homogeneous derivations of polynomial degree $p$, then
$\Derivation{S} = \bigoplus_{p \in \Ganz} \Derivation{S}_p$.

\begin{definition}
Let $\A$ be an arrangement of hyperplanes in $V$. Then for $H \in \A$ we fix $\alpha_H \in V^*$ with $H = \Kern{\alpha_H}$.
A \emph{defining polynomial} $\PolyA{\A}$ is given by $\PolyA{\A} := \prod_{H \in \A} \alpha_H \in S$.

The \emph{module of $\A$-derivations} of $\A$ is defined by
\begin{equation*}
	\DerA{\A} := \DerA{\PolyA{\A}} := \{ \theta \in \Derivation{S} \mid \theta(\PolyA{\A}) \in \PolyA{\A}S \}.
\end{equation*}

We say that $\A$ is \emph{free} if the module of $\A$-derivations is a free $S$-module.
\end{definition}

If $\A$ is a free arrangement, let $\{ \theta_1, \ldots, \theta_\ell \}$ be a homogeneous basis for $\DerA{\A}$.
Then the polynomial degrees of the $\theta_i$, $i \in \{1,\ldots,\ell\}$, are called the \emph{exponents} of $\A$.
We write $\expAA{\A} := \{\{\pdeg{\theta_1},\ldots,\pdeg{\theta_\ell}\}\}$, where the notation $\{\{ * \}\}$ emphasizes the fact,
that $\expAA{\A}$ is a multiset in general.
The multiset $\expAA{\A}$ is uniquely determined by $\A$, see also \cite[Def.\ 4.25]{orlik1992arrangements}.

If $\A$ is free with exponents $\expA{\A}{b_1,\ldots,b_\ell}$, then by \cite[Thm.\ 4.23]{orlik1992arrangements}:
\begin{equation}
	\sum_{i=1}^\ell b_i = \Betrag{\A}. \label{Sum_exp}
\end{equation}

The following proposition shows that the product construction mentioned before is compatible with the notion
of freeness:

\begin{proposition}[{\cite[Prop.~4.28]{orlik1992arrangements}}]\label{prop:Prod_free}
Let $(\A_1,V_1)$ and $(\A_2,V_2)$ be two arrangements.
The product arrangement $(\A_1 \times \A_2, V_1 \oplus V_2)$ is free if and only if both $(\A_1,V_1)$ and 
$(\A_2,V_2)$ are free. In this case
\begin{equation*}
\expAA{\A_1 \times \A_2} = \expAA{\A_1} \cup \expAA{\A_2}.
\end{equation*}
\end{proposition}

Throughout our exposition we will frequently use the following important results about free arrangements.

\begin{theorem}[Addition-Deletion {\cite[Thm.\ 4.51]{orlik1992arrangements}}]\label{thm:Addition_Deletion}
Let $\A$ be a hyperplane arrangement and $\A \neq \emptA{\ell}$. Let $(\A,\A',\A'')$ be a triple.
Any two of the following statements imply the third:
\begin{eqnarray*}
\A \text{ is free with } \expAA{\A} &=& \{\{b_1,\ldots,b_{l-1},b_l\}\}, \\
\A' \text{ is free with } \expAA{\A'} &=& \{\{b_1,\ldots,b_{l-1},b_l-1\}\}, \\
\A'' \text{ is free with } \expAA{\A''} &=& \{\{b_1,\ldots,b_{l-1}\}\}.
\end{eqnarray*}
\end{theorem}

Choose a hyperplane $H_0 = \ker{\alpha_0} \in \A$. Let $\bar{S} = S / \alpha_0 S$.
If $\theta \in \DerA{\A}$, then $\theta(\alpha_0 S) \subseteq \alpha_0 S$.
Thus we may define $\Abb{\bar{\theta}}{\bar{S}}{\bar{S}}$ by $\bar{\theta}(f + \alpha_0 S) = \theta(f) + \alpha_0 S$.
Then $\bar{\theta} \in \DerA{\An{A''}}$, \cite[Def.\ 4.43, Prop.\ 4.44]{orlik1992arrangements}.

\begin{theorem}[{\cite[Thm.\ 4.46]{orlik1992arrangements}}]\label{thm:A_AoH_exp}
Suppose $\A$ and $\A'$ are free arrangements with $\A' := \A \setminus \{ H_0 \}$, $H_0 := \ker{\alpha_0}$. 
Then there is a basis $\{ \theta_1,\ldots,\theta_\ell\}$ for $\DerA{\A'}$ such that
\begin{enumerate}
\item[(1)] $\{ \theta_1,\ldots,\theta_{i-1},\alpha_0\theta_i,\theta_{i+1},\ldots,\theta_\ell\}$ is a basis for $\DerA{\A}$,
\item[(2)]  $\{ \bar{\theta}_1,\ldots,\bar{\theta}_{i-1},\bar{\theta}_{i+1},\ldots,\bar{\theta}_\ell \}$ is a basis for $\DerA{\A''}$.
\end{enumerate}
\end{theorem}

\begin{theorem}[Factorization {\cite[Thm.\ 4.137]{orlik1992arrangements}}]\label{thm:A_free_factZ}
If $\A$ is a free arrangement with $\expA{\A}{b_1,\ldots,b_\ell}$, then
\begin{equation*}
\CharPolyA{\A} = \prod_{i=1}^\ell (t - b_i).
\end{equation*}
\end{theorem}

A very recent and remarkable result is due to Abe which connects the division of characteristic polynomials with freeness:

\begin{theorem}[Division theorem {\cite[Thm.~1.1]{AbeDivFree2015}}]\label{thm:Div_Thm}
Let $\A$ be a hyperplane arrangement and $\A \neq \emptA{\ell}$. Assume that there is a hyperplane $H \in \A$ such that
$\CharPolyA{\A^H}$ divides $\CharPolyA{\A}$ and $\A^H$ is free. Then $\A$ is free.
\end{theorem}

\subsection{Inductively, recursively and divisionally free arrangements}
Theorem \ref{thm:Addition_Deletion} motivates the following two definitions of classes of free arrangements.

\begin{definition}[{\cite[Def.\ 4.53]{orlik1992arrangements}}] \label{def:IF}
The class $\IFC$ of \emph{inductively free} arrangements is the smallest class of arrangements which satisfies
\begin{enumerate}
\item The empty arrangement $\emptA{\ell}$ of rank $\ell$ is in $\IFC$ for $\ell \geq 0$,
\item if there exists a hyperplane $H_0 \in \A$ such that $\A'' \in \IFC$, $\A' \in \IFC$, and
$\expAA{\A''} \subset \expAA{\A'}$, then $\A$ also belongs to $\IFC$. 
\end{enumerate}
\end{definition}

\begin{example}\label{ex:SS_IF}
All supersolvable arrangements are inductively free by \cite[Thm.\ 4.58]{orlik1992arrangements}.
\end{example}

\begin{definition}[{\cite[Def.\ 4.60]{orlik1992arrangements}}] \label{def:RF}
The class $\RFC$ of \emph{recursively free} arrangements is the smallest class of arrangements which satisfies
\begin{enumerate}
\item The empty arrangement $\emptA{\ell}$ of rank $\ell$ is in $\RFC$ for $\ell \geq 0$,
\item if there exists a hyperplane $H_0 \in \A$ such that $\A'' \in \RFC$, $\A' \in \RFC$, and
$\expAA{\A''} \subset \expAA{\A'}$, then $\A$ also belongs to $\RFC$,
\item if there exists a hyperplane $H_0 \in \A$ such that $\A'' \in \RFC$, $\A \in \RFC$, and
$\expAA{\A''} \subset \expAA{\A}$, then $\A'$ also belongs to $\RFC$. 
\end{enumerate}
\end{definition}

Note that we have:
\begin{equation*}
\IFC \subsetneq \RFC \subsetneq \{ \text{ \emph{free arrangements} } \},
\end{equation*}
where the properness of the last inclusion was only recently established by Cuntz and Hoge in \cite{MR3272729}.

Furthermore, similarly to the class $\IFC$ of inductively free arrangements, Theorem \ref{thm:Div_Thm} motivates the following class 
of free arrangements:

\begin{definition}[{\cite[Def.~4.3]{AbeDivFree2015}}]\label{def:DF}
The class $\DFC$ of \emph{divisionally free} arrangements is the smallest class of arrangements which satisfies
\begin{enumerate}
\item If $\A$ is an $\ell$-arrangement, $\ell \leq 2$, or $\A = \emptA{\ell}$, $\ell \geq 3$,
then $\A$ belongs to $\DFC$,
\item if there exists a hyperplane $H_0 \in \A$ such that $\A'' \in \DFC$ and
$\CharPolyA{\A^{H_0}} \mid \CharPolyA{\A}$, then $\A$ also belongs to $\DFC$. 
\end{enumerate}
\end{definition}

Abe showed that the new class of divisionally free arrangements properly contains the class of inductively free
arrangements:
\begin{equation*}
\IFC \subsetneq \DFC,
\end{equation*}
by \cite[Thm.~1.6]{AbeDivFree2015}.
He conjectured that there are arrangements which are divisionally free but not recursively free.
Our classification of recursively free reflection arrangements in this paper provides examples to confirm his conjecture
(see Theorem \ref{thm:ConjAbe} and Section \ref{sec:Abes_Conjecture}). 

The next easy lemma will be useful to disprove the recursive freeness of a given arrangement:

\begin{lemma}\label{lem:A_u_H}
Let $\A$ and $\A' = \A \setminus \{H\}$ be free arrangements
and $L := \LA{\A'}$.
Let $P_H := \{ X \in L_2 \mid X \subseteq H \} = \A'' \cap L_2$,
then $\sum_{X \in P_H} (\Betrag{\A'_{X}} - 1) \in \expAA{\A'}$,
and if $\A'$ is irreducible then $\sum_{X \in P_H} (\Betrag{\A'_{X}} - 1) \neq 1$.
\end{lemma}

\begin{proof}
By Theorem \ref{thm:A_AoH_exp} and (\ref{Sum_exp}) there is a $b \in \expAA{\A'}$, such that 
$\Betrag{\A^{H}} = \Betrag{\A'} - b$ and if $\A'$ is irreducible, then $b \neq 1$.
Since $\Betrag{\A^{H}} =  \Betrag{\A'} - \sum_{X \in P_H} (\Betrag{\A'_X} -1)$,
the claim directly follows.
\end{proof}

The next two results are due to Hoge, R\"ohrle, and Schauenburg, \cite{HRS15}.

\begin{proposition}[{\cite[Thm.\ 1.1]{HRS15}}]\label{prop:Arf_AXrf}
Let $\A$ be a recursively free arrangement and $X \in \LA{\A}$.
Then $\A_X$ is recursively free.
\end{proposition}

Hoge and R\"ohrle have shown in \cite[Prop.\ 2.10]{2012arXiv1208.3131H}
that the product construction is compatible with the notion of inductively free arrangements.

The following refines the statement for recursively free arrangements:

\begin{proposition}[{\cite[Thm.\ 1.2]{HRS15}}]\label{prop:Prod_RF}
Let $(\A_1,V_1), (\A_2,V_2)$ be two arrangements. Then  $\A = (\A_1 \times \A_2,V_1 \oplus V_2)$
is recursively free if and only if both $(\A_1,V_1)$ and $(\A_2,V_2)$ are recursively free and in that case the multiset of 
exponents of $\A$ is given by $\expAA{\A} = \expAA{\A_1} \cup \expAA{\A_2}$.
\end{proposition}

\subsection{Reflection arrangements}

Let $V = \Komplex^\ell$ be a finite dimensional complex vector space.
An element $s \in \GeneralLinearGroup{V}$ of finite order with fixed point set
$V^s = \{ x \in V \mid sx =x \} =H_s$ a hyperplane in $V$ is called a \emph{reflection}.
A finite subgroup $W \leq \GeneralLinearGroup{V}$ which is generated by reflections 
is called a \emph{finite complex reflection group}.

The finite complex reflection groups were classified by Shephard and Todd, \cite{ST_1954_fcrg}. 

Let $W \leq \GeneralLinearGroup{V}$ be a finite complex reflection group acting on the vector space $V$.
The \emph{reflection arrangement} $(\A(W),V)$ is the arrangement of hyperplanes
consisting of all the reflecting hyperplanes of reflections of $W$.

Terao \cite{Terao1980} has shown that each reflection arrangement is free, see also \cite[Prop.~6.59]{orlik1992arrangements}. 

The complex reflection group $W$ is called \emph{reducible} if $V = V_1 \oplus V_2$ where $V_i$ are stable under $W$.
Then the restriction $W_i$ of $W$ to $V_i$ is a reflection group in $V_i$.
In this case the reflection arrangement $(\A(W),V)$ is the product of the two 
reflection arrangements $(\A(W_1),V_1)$ and $(\A(W_2),V_2)$.
The complex reflection group $W$ is called \emph{irreducible} if it is not reducible,
and then the reflection arrangement $\A(W)$ is irreducible.

For later purposes, we now look at the action of a finite complex reflection group $W$ on its
associated reflection arrangement $\A(W)$ and (reflection) subarrangements $\A(W') \subseteq \A(W)$
corresponding to reflection subgroups $W' \leq W$.

Let $W$ be a finite complex reflection group and $\A := \A(W)$.
Then $W$ acts on the set $\SAA := \{ \B \mid \B \subseteq \A \}$ of subarrangements of $\A$
by $w.\B = \{ w.H \mid H \in \B\}$ for $\B \in \SAA$.
The \emph{(setwise) stabilizer} $S_\B$ of $\B$ in $W$ ist defined by $S_\B = \{ w \in W \mid w.\B=\B \}$.
We denote by $W.\B = \{w.\B \mid w \in W\} \subseteq \SAA$ the orbit of $\B$ under $W$.

The following lemma is similar to a statement from \cite[Lem.~6.88]{orlik1992arrangements}.

\begin{lemma}\label{lem:Action_W_RSA}
Let $W$ be a finite complex reflection group, $\A := \A(W)$, and $\SAA = \{ \B \mid \B \subseteq \A \}$.
Let $\B := \A(W') \in \SAA$ be a reflection subarrangement for a reflection subgroup $W' \leq W$.
Then $S_\B = N_W(W')$ and $\Betrag{W.\B} =  
\Betrag{W:S_\B} = \Betrag{W:N_W(W')}$.
\end{lemma}

\begin{proof}
Let $W$, $W'$, $\A$, and $\B$ be as above.
Let $S_\B$ be the stabilizer of $\B$ in $W$.
It is clear by the Orbit-Stabilizer-Theorem that $\Betrag{W.\B} = \Betrag{W:S_\B}$.
Let $H_r \in W'$ for a reflection $r \in W'$, then $w.H_r = H_{w^{-1}rw}$.
So we have
\begin{align*}
S_\B  = & \quad \{ w \in W \mid w.H_r \in \B \text{ for all reflections } r \in W' \} \\
= & \quad \{ w \in W \mid w^{-1}rw \in W' \text{ for all reflections } r \in W'\} \\
= & \quad N_W(W').
\end{align*}
The last equality is because $W'$ is by definition generated by the reflections it contains
and the group normalizing all generators of $W'$ is the normalizer of $W'$.
\end{proof}

The following theorem proved by Barakat, Cuntz, Hoge and R\"ohrle, which provides a classification of all inductively free
reflection arrangments, is our starting point for inspecting the recursive freeness of reflection arrangements:

\begin{theorem}[{\cite[Thm.~1.1]{2012arXiv1208.3131H}, \cite[Thm.~5.14]{MR2854188}}]\label{thm:RA_IF}
For $W$ a finite complex reflection group, the reflection arrangement $\A(W)$ is inductively free
if and only if $W$ does not admit an irreducible factor isomorphic to a monomial group $\crgM{r}{r}{\ell}$ for $r,\ell \geq 3$,
$\crg{24}$, $\crg{27}$, $\crg{29}$, $\crg{31}$, $\crg{33}$, or $\crg{34}$.  
\end{theorem}

Thus, to prove Theorem \ref{thm:RFcrArr}, we only have to check 
the non\--in\-duc\-tively free cases from Theorem \ref{thm:RA_IF} since inductive freeness implies recursive freeness.

\section{Proof of Theorem \ref{thm:RFcrArr}}

Thanks to Proposition \ref{prop:Prod_RF}, the proof of Theorem \ref{thm:RFcrArr}  reduces to the case
when $\A(W)$ respectively $W$ are irreducible. 
We consider the different irreducible reflection arrangements, provided by Theorem \ref{thm:RA_IF},
which are not inductively free, in turn.

\subsection{The reflection arrangements $\A(\crgM{r}{r}{\ell})$, $r,\ell \geq3$}\label{sec3}

For an integer $r \geq 2$ let $\theta = \exp{(2\pi i / r)}$, and $C(r)$ the cyclic group generated by $\theta$.
The reflection arrangement $\A(W)$ with $W = \crgM{r}{r}{\ell}$ contains the hyperplanes
\begin{equation*}
H_{i,j}(\zeta) := \Kern{ x_i - \zeta x_j}, 
\end{equation*} 
with $i,j \leq \ell$ and $i \neq j$, $\zeta \in C(r)$,
and if $W$ is the full monomial group $\crgM{r}{1}{\ell}$,
then $\A(\crgM{r}{1}{\ell}$
additionally contains the coordinate hyperplanes $E_i := \Kern{x_i}$, \cite[Ch.~6.4]{orlik1992arrangements}.

To show that the reflection arrangements $\A(\crgM{r}{r}{\ell})$ for $r,\ell \geq 3$ are
recursively free, we need the intermediate arrangements $\A_\ell^k(r)$ with 
$\A(\crgM{r}{r}{\ell}) \subseteq \A_\ell^k(r) \subseteq  \A(\crgM{r}{1}{\ell})$.
They are defined as follows:
\begin{equation*}
\A_\ell^k(r) := \A(\crgM{r}{r}{\ell}) \dot{\cup} \{ E_1,\ldots,E_k\},
\end{equation*}
and their defining polynomial is given by
\begin{equation*}
Q(\A_\ell^k(r)) = x_1\cdots x_k \prod_{\substack{1 \leq i < j \leq \ell \\ 0 \leq n < r}} (x_i - \zeta^n x_j).
\end{equation*}

The following result by Amend, Hoge and R\"ohrle immediately implies the recursive freeness of $\A(\crgM{r}{r}{\ell})$,
for $r,\ell \geq 3$.

\begin{theorem}[{\cite[Thm.~3.6]{MR3250448}}]\label{thm:Akl_rf}
Suppose $r \geq 2$, $\ell \geq 3$ and $0 \leq k \leq \ell$.
Then $\A_\ell^k(r)$ is recursively free.
\end{theorem}

\begin{corollary}\label{coro:Grrl_RF}
Let $W$ be the finite complex reflection group $W = \crgM{r}{r}{\ell}$.
Then the reflection arrangement $\A := \A(W)$ is recursively free.
\end{corollary}
\begin{proof}
We have $\A \cong \A_\ell^0(r)$ and  by
Theorem \ref{thm:Akl_rf}, $\A_\ell^0(r)$ is recursively free.
\end{proof}

\subsection{The reflection arrangement $\A(\crg{24})$}\label{subsec:A24}

We show that the reflection arrangement of the finite complex reflection group $\crg{24}$ is
recursively free by constructing a so called supersolvable resolution for the arrangement, (see also \cite[Ch.~3.6]{Zielger87}), 
and making sure that in each addition-step 
of a new hyperplane the resulting arrangements and restrictions are free with suitable exponents.
As a supersolvable arrangement is always inductively free (Example \ref{ex:SS_IF}), 
it follows that $\A(\crg{24})$ is recursively free.

\begin{lemma}
Let $W$ be the complex reflection group $W = \crg{24}$.
Then the reflection arrangement $\A = \A(W)$ is recursively free.
\end{lemma}

\begin{proof}
Let $\omega := -\frac{1}{2}(1+i\sqrt{7})$, then the reflecting hyperplanes of $\A$ can be defined by the following
linear forms (see also \cite[Ch.\ 7, 6.2]{lehrer2009unitary}):
\begin{align*}
\A = \quad \{ \, &(1,0,0)^\perp,(0,1,0)^\perp,(0,0,1)^\perp,(1,1,0)^\perp,(-1,1,0)^\perp, \\
	& (1,0,1)^\perp,(-1,0,1)^\perp,(0,1,1)^\perp,(0,-1,1)^\perp,(\omega,\omega,2)^\perp, \\
	& (-\omega,\omega,2)^\perp,(\omega,-\omega,2)^\perp,(-\omega,-\omega,2)^\perp,(\omega,2,\omega)^\perp, \\
	& (-\omega,2,\omega)^\perp,(\omega,2,-\omega)^\perp,(-\omega,2,-\omega)^\perp,(2,\omega,\omega)^\perp, \\
	& (2,-\omega,\omega)^\perp,(2,\omega,-\omega)^\perp,(2,-\omega,-\omega)^\perp \, \}.
\end{align*}
The exponents of $\A$ are $\expA{\A}{1,9,11}$.

If we define
\begin{align*}
\{ H_1,\ldots,H_{12} \} := \quad \{ \, &(\omega^2,\omega,0)^\perp,(-\omega^2,\omega,0)^\perp,(\omega,\omega^2,0)^\perp \\
&(-\omega,\omega^2,0)^\perp,(2-\omega,\omega,0)^\perp, (-2+\omega,\omega,0)^\perp, \\
&(\omega,2-\omega,0)^\perp,(-\omega,2-\omega,0)^\perp,(\omega,2,0)^\perp, \\
&(-\omega,2,0)^\perp, (2,\omega,0)^\perp,(-2,\omega,0)^\perp \, \},
\end{align*}
and the arrangements $\A_j := \A \dot{\cup} \{ H_1,\ldots,H_j \}$ for $1 \leq j \leq 12$,
then
\begin{equation*}
X = (1,0,0)^\perp \cap (0,1,0)^\perp \cap (1,1,0)^\perp \cap (-1,1,0)^\perp \cap_{j=1}^{12} H_j \in \LA{\A_{12}}
\end{equation*}
is a rank $2$ modular element, and $\A_{12}$ is supersolvable. 
In each step, $\A_j$ is free, $\A_j^{H_j}$ is inductively free (since $\A_j^{H_j}$ is a $2$-arrangement),
and $\expAA{\A_j^{H_j}} \subseteq \expAA{\A_j}$.
The exponents of the arrangements $\A_j$ and $\A_j^{H_j}$ are listed in Table \ref{tab_exp_Ai}.
\begin{table}
\begin{tabular}{l l l}
\hline
$j$ & $\expAA{\A_j}$ & $\expAA{\A_j^{H_j}}$ \\
\hline
\hline
1 	& 1,10,11 & 1,11 \\
2 	& 1,11,11 & 1,11 \\
3 	& 1,11,12 & 1,11 \\
4 	& 1,11,13 & 1,11 \\
5 	& 1,12,13 & 1,13 \\
6 	& 1,13,13 & 1,13 \\
7 	& 1,13,14 & 1,13 \\
8 	& 1,13,15 & 1,13 \\
9 	& 1,14,15 & 1,15 \\
10 	& 1,15,15 & 1,15 \\
11 	& 1,15,16 & 1,15 \\
12 	& 1,15,17 & 1,15 \\
\hline
\\
\end{tabular}
\caption{The exponents of the free arrangements $\A_j$ and $\A_j^{H_j}$.}\label{tab_exp_Ai}
\end{table}

Since by Example \ref{ex:SS_IF} a supersolvable arrangement is inductively free, $\A$ is recursively free. 
\end{proof}

We found the set of hyperplanes $\{ H_1,\ldots,H_{12} \}$ 
by ``connecting'' a suitable $X \in \LA{\A}_2$ to other $Y \in \LA{\A}_2$
via addition of new hyperplanes such that $X$ becomes a modular element in the resulting intersection lattice,
subject to each addition of a new hyperplane results in a free arrangement,
(compare with \cite[Ex.~4.59]{orlik1992arrangements}).

\subsection{The reflection arrangement $\A(\crg{27})$}\label{subsec:A27}

In \cite[Remark~3.7]{MR3272729} Cuntz and Hoge have shown that the reflection arrangement $\A(\crg{27})$ is not
recursively free. In particular they have shown that there is no hyperplane which can be added or removed from $\A(\crg{27})$
resulting in a free arrangement.

\subsection{The reflection arrangements $\A(\crg{29})$ and $\A(\crg{31})$}\label{subsec:A29_A31}

In \cite[Lemma~3.5]{2012arXiv1208.3131H} Hoge and R\"ohrle settled the case that the reflection arrangement $\A(\crg{31})$ 
of the exceptional finite complex reflection group $\crg{31}$ is not inductively free by testing several cases with the computer.

In this part we will see, that the reflection arrangement $\A(\crg{31})$ is additionally not recursively free and
as a consequence the closely related reflection subarrangement $\A(\crg{29})$ is also not recursively free.
Furthermore we obtain a new computer-free proof, that $\A(\crg{31})$ is not inductively free.

\begin{theorem}\label{thm:G31_G29_nRF}
Let $\A = \A(W)$ be the reflection arrangement with $W$ isomorphic to one of the finite complex reflection groups 
$\crg{29}$, $\crg{31}$. Then $\A$ is not recursively free.
\end{theorem}

We will prove the theorem in two parts.

In the first part, we will characterize certain free subarrangements of $\A(\crg{31})$ which we can obtain out of $\A(\crg{31})$ by
successive deletion of hyperplanes such that all the arrangements in between are also free.
We call such arrangements \emph{\FFSAsp}.
Then we will investigate the relation between the two reflection arrangements $\A(\crg{29})$ and $\A(\crg{31})$,
and obtain that $\A(\crg{29})$ is the smallest of these \FFSAs of $\A(\crg{31})$.
This yields a new proof, that $\A(\crg{31})$ is not inductively free (since inductive freeness implies that the empty arrangement
is a \FFSAp).

In the second part, we will show that if $\tilde{\A}$ is a \FFSA of $\A(\crg{31})$,
there is no possible way to obtain a free arrangement out of $\tilde{\A}$
by adding a new hyperplane which is not already contained in $\A(\crg{31})$.

This will conclude the proof of Theorem \ref{thm:G31_G29_nRF}.

\begin{definition}\label{def:ArrG31}
Let $i = \sqrt{-1}$. The arrangement $\A(\crg{31})$ can be defined as the union of the following collections of hyperplanes:
\begin{align}%\label{ArrG31}
\begin{split}
\A(\crg{31})  := \, %
	& \{ \Kern{x_p- i^k x_q} \mid 0 \leq k \leq 3, 1 \leq p < q \leq 4 \}  \, \dot{\cup} \\
	& \{ \alpha^\perp \mid \alpha \in \crgM{4}{4}{4}.(1,1,1,1) \}  \, \dot{\cup} \\
	& \{ (1,0,0,0)^\perp, (0,1,0,0)^\perp, (0,0,1,0)^\perp, (0,0,0,1)^\perp \}   \,  \dot{\cup}  \\
	& \{ \alpha^\perp \mid \alpha \in \crgM{4}{4}{4}.(-1,1,1,1) \}. 
\end{split}
\end{align}

The first set contains the hyperplanes of the reflection arrangement $\Arr{\crgM{4}{4}{4}}$.
The second and the last set contain the hyperplanes defined by the linear forms in orbits of the group $\crgM{4}{4}{4}$.
The union of the first and the second set gives the $40$ hyperplanes of the reflection arrangement $\A(\crg{29})$. 
In particular, $\A(\crg{29}) \subseteq \A(\crg{31})$, compare with \cite[Ch.\ 7, 6.2]{lehrer2009unitary}.
\end{definition}

\subsubsection{The free filtration subarrangements of $\A(\crg{31})$} 

In this subsection we characterize certain free subarrangements of $\A(\crg{31})$ which we can obtain by successively removing
hyperplanes from $\A(\crg{31})$, the so called \emph{free filtration subarrangements}.
We will use this characterization in Subsection \ref{subsubsec:A29_A31_nRF_e} 
to prove Theorem \ref{thm:G31_G29_nRF}.
Furthermore, along the way, we obtain another (computer-free) proof that the arrangement 
$\A(\crg{31})$ cannot be inductively free 
(recall Definition \ref{def:IF}) without checking all the cases for a possible inductive 
chain but rather by examining the intersection lattices of certain subarrangements and using the fact, that $\A(\crg{29})$ 
plays a ``special'' role among the free filtration subarrangements of $\A(\crg{31})$.

\begin{definition}\label{def:FFSA}
Let $\A$ be a free $\ell$-arrangement and $\tilde{\A} \subseteq \A$ a free sub\-arrangement.
A strictly decreasing sequence of free arrangements
\begin{equation*}
\A = \A_0 \supsetneq \A_1 \supsetneq \ldots \supsetneq \A_{n-1} \supsetneq \A_n = \tilde{\A}
\end{equation*}
is called a \emph{(finite) free filtration} from $\A$ to $\tilde{\A}$ if $\Betrag{\A_i} = \Betrag{\A}-i$ for each $i$.
If there exists a (finite) free filtration from $\A$ to $\tilde{\A}$, 
we call $\tilde{\A}$ a \emph{\FFSAp}.
\end{definition}

The notion of free filtration was first introduced by Abe and Terao in \cite{Abe2015}. 

Note that, since all the subarrangements $\A_i$ in the definition are free, 
with Theorem \ref{thm:A_AoH_exp} the restrictions $\A_{i-1}^{H_i}$ are free and
we automatically have $\expAA{\A_{i-1}^{H_i}} \subseteq \expAA{\A_{i-1}}$
and $\expAA{\A_{i-1}^{H_i}} \subseteq \expAA{\A_{i}}$.

If $\A$ is an inductively free $\ell$-arrangement, then $\emptA{\ell}$ is a free filtration subarrangement.

The main result of this subsection is the following proposition which we will prove in several steps divided into
some lemmas.

\begin{proposition}\label{prop:G31_NIF}
Let $\A := \A(\crg{31})$ be the reflection arrangement of the finite complex reflection group $\crg{31}$.
Let $\tilde{\A}$ be a smallest (w.r.t. the number of hyperplanes)
free filtration subarrangement.
Then $\tilde{\A} \cong \A(\crg{29})$.
In particular $\A$, $\A(\crg{29})$ and all other free filtration subarrangements $\tilde{\A} \subseteq \A$ 
are not inductively free.
\end{proposition}

To prove Proposition \ref{prop:G31_NIF}, we will characterize all free filtration subarrangements of $\A(\crg{31})$
by certain combinatorial properties of their intersection lattices.

The next lemma gives a sufficient condition for $\tilde{\A} \subseteq \A(\crg{31})$ being a free filtration subarrangement.
With an additional assumption on $\Betrag{\tilde{\A}}$, this condition is also necessary. 

\begin{lemma}\label{lem:SubsetN_FFSA}
Let $\An{N} \subseteq \A :=\A(\crg{31})$ be a subcollection of hyperplanes
and $\tilde{\A} := \A \setminus \An{N}$. If $\An{N}$ satisfies 
\begin{equation}\label{eq:SubsetN_FFSA}
	\tag{$\ast$} \bigcup_{X \in \LAq{\An{N}}{2}} X \subseteq \bigcup_{H \in {\tilde{\A}}} H  \text{,}
\end{equation} 
then $\tilde{\A} \subseteq \A$ is a free filtration subarrangement, 
with exponents $\expA{\tilde{\A}}{1,13,17,29-\Betrag{\An{N}}}$.

If furthermore $\Betrag{\An{N}} \leq 13$, then $\tilde{\A} \subseteq \A$ is a free filtration subarrangement if and only if
$\An{N}$ satisfies (\ref{eq:SubsetN_FFSA}).
\end{lemma}
\begin{proof}
We proceed by induction on $\Betrag{\An{N}}$.

We use the fact, that $\crg{31}$ acts transitively on the hyperplanes of $\A$.
In particular, all the 3-arrangements $\A^H$ for $H \in \A$ are isomorphic and furthermore they are 
free with exponents $\expA{\A^{H}}{1,13,17}$ (see \cite[App.\ C and App.\ D]{orlik1992arrangements}).

First let $\An{N} = \{ H\}$ consist of only a single hyperplane.
Since $\A$ is free with exponents $\expA{\A}{1,13,17,29}$,
the arrangement $\tilde{\A}= \A'$ is just a deletion with respect to $H$, hence free
by Theorem \ref{thm:Addition_Deletion}, and $\tilde{\A}$ is a \FFSA
with $\expA{\tilde{\A}}{1,13,17,28}$. 

With $\An{N}$, each subcollection $\An{N'} = \An{N} \setminus \{K\}$, for a
$K \in \An{N}$, fulfills the assumption of the lemma. 
By the induction hypotheses $\B = \A \setminus \An{N}'$ is a \FFSA with 
$\expA{\B}{1,13,17,29-\Betrag{\An{N}'}} = \{\{1,13,17,29-\Betrag{\An{N}}+1\}\}$.
Now condition (\ref{eq:SubsetN_FFSA}) %$\bigcup_{X \in L_2} X \subseteq \bigcup_{H \in {\tilde{\A}}} H$ 
just means that
$\Betrag{\B^K} = 31$, so $\An{B}^K \cong \A^H$ for any $H \in \A$ and is free with $\expA{\An{B}^K}{1,13,17}$.
Hence, again by Theorem \ref{thm:Addition_Deletion}, the deletion 
$\An{B}' = \An{B} \setminus \{K\}$ is free
and thus $\tilde{\A} = \A \setminus \An{N} = \An{B}'$ is a \FFSA
with $\expA{\tilde{\A}}{1,13,17,29-\Betrag{\An{N}}}$.

Finally, let $\tilde{\A} = \A \setminus \An{N}$ be a \FFSA with $\Betrag{\An{N}} \leq 13$.
For an associated \FF $\A = \A_0 \supsetneq \ldots \supsetneq \A_n = \tilde{\A}$
with say $\A_i = \A_{i-1}' = \A_{i-1} \setminus \{H_i\}$ for $1 \leq i \leq n$,
we have $\Betrag{\A_{i-1}^{H_i}} = 31$.
So $H_i \cap H_j \subseteq K$, $j < i$, for a $K \in \A_i$ and  for $i=n$ this is condition (\ref{eq:SubsetN_FFSA}).
\end{proof}

Before we continue with the characterization of the \FFSAsp, we give a helpful partition of
the reflection arrangement $\A(\crg{31})$:

\begin{lemma}\label{lem:Part_AG31}
Let $\A = \A(\crg{31})$. 
There are exactly $6$ subcollections $M_1,\ldots,M_6 \subseteq \A$, such that
$\A \setminus M_i \cong \A(\crg{29})$, $M_i \cap M_j \cong \A(A_1^4)$ and $M_i \cap M_j \cap M_k = \emptyset$ for
$1\leq i <j<k\leq 6$.
Thus we get a partion of $\A$ into $15$ disjoint subsets $\{ M_i \cap M_j \mid 1\leq i < j \leq 6 \} =:\F$
on which $\crg{31}$ acts transitively.
\end{lemma}

\begin{figure}[h!]
\setlength{\unitlength}{10mm}
\begin{picture}(12,7)

	\put(2,6){\line(1,0){10}}
	\put(2,5){\line(1,0){10}}
	\put(4,4){\line(1,0){8}}
	\put(6,3){\line(1,0){6}}
	\put(8,2){\line(1,0){4}}
	\put(10,1){\line(1,0){2}}

	\put(12,6){\line(0,-1){5}}
	\put(10,6){\line(0,-1){5}}
	\put(8,6){\line(0,-1){4}}
	\put(6,6){\line(0,-1){3}}
	\put(4,6){\line(0,-1){2}}
	\put(2,6){\line(0,-1){1}}

	\put(2.2,5.4){$M_1 \cap M_2$}
	\put(4.2,5.4){$M_1 \cap M_3$}
	\put(6.2,5.4){$M_1 \cap M_4$}
	\put(8.2,5.4){$M_1 \cap M_5$}
	\put(10.2,5.4){$M_1 \cap M_6$}

	\put(4.2,4.4){$M_2 \cap M_3$}
	\put(6.2,4.4){$M_2 \cap M_4$}
	\put(8.2,4.4){$M_2 \cap M_5$}
	\put(10.2,4.4){$M_2 \cap M_6$}

	\put(6.2,3.4){$M_3 \cap M_4$}
	\put(8.2,3.4){$M_3 \cap M_5$}
	\put(10.2,3.4){$M_3 \cap M_6$}

	\put(8.2,2.4){$M_4 \cap M_5$}
	\put(10.2,2.4){$M_4 \cap M_6$}

	\put(10.2,1.4){$M_5 \cap M_6$}

	%\put(4.7,2.4){$\crg{31}$\Huge{$\circlearrowright$}}
	%\linethickness{0.2mm}

	\put(6.05,4.95){\line(1,0){1.9}}
	\put(6.05,4.05){\line(1,0){1.9}}
	\put(6.05,4.95){\line(0,-1){0.9}}
	\put(7.95,4.95){\line(0,-1){0.9}}

	%\put(1,5.5){\circle{0.8}}
	%\put(0.75,5.4){$M_1$}
	%\put(1.4,5.5){\line(1,0){10.6}}

	\put(3,4.5){\circle{0.8}}
	\put(2.75,4.4){$M_2$}
	\put(3,6){\line(0,-1){1.1}}
	\put(3.4,4.5){\line(1,0){8.6}}
%
%	\put(5,3.5){\circle{0.8}}
%	\put(4.75,3.4){$M_3$}
%	\put(5,6){\line(0,-1){2.1}}
%	\put(5.4,3.5){\line(1,0){6.6}}
%
	\put(7,2.5){\circle{0.8}}
	\put(6.75,2.4){$M_4$}
	\put(7,6){\line(0,-1){3.1}}
	\put(7.4,2.5){\line(1,0){4.6}}
%
%	\put(9,1.5){\circle{0.8}}
%	\put(8.75,1.4){$M_5$}
%	\put(9,6){\line(0,-1){4.1}}
%	\put(9.4,1.5){\line(1,0){2.6}}

\end{picture}
\caption{The partition of $\A$ into $15$ disjoint subsets $\F = \{M_i \cap M_j$, $1 \leq i < j \leq 6\}$, each consisting of $4$ hyperplanes.}%
\label{fig:Part_A_15}
\end{figure}
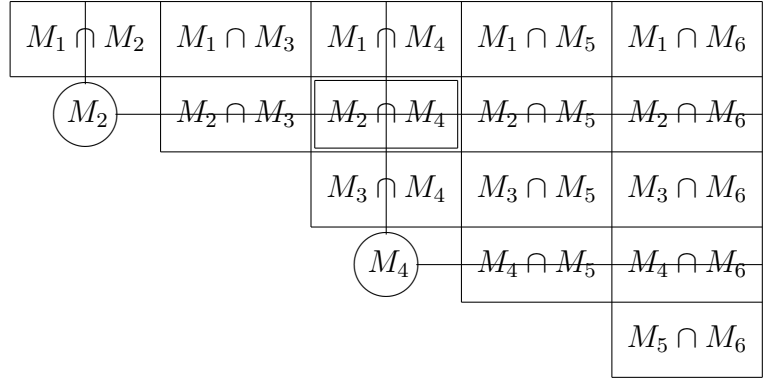

\begin{proof}
Let $W := \crg{31}$ and $W' := \crg{29} \leq W$.
Then $N_W(W') = W'$ and $\Betrag{W:W'} = 6$, so with Lemma \ref{lem:Action_W_RSA} there are exactly $6$
subarrangements, say $\B_1,\ldots,\B_6$ with $\B_i \cong \A(W') \subseteq \A$, 
(respectively 6 conjugate reflection subgroups of $W$ isomorphic to $W'$).
Now we get the $M_i$ as $M_i = \A \setminus \B_i$ and 
in particular the corresponding action of $W$ on the subcollections is transitive.

To see the claim about their intersections we look at the different orbits of reflection subgroups of $W$ on $\A$
acting on hyperplanes.
First $W'$ has 2 orbits $\An{O}_1 = \A(W')$, and $\An{O}_2 = \A \setminus \A(W') = M_i$ for an $i \in\{1,\ldots,6\}$.
Similarly a subgroup $\tilde{W'} = g^{-1}W'g \neq W'$ conjugate to $W'$  has also 2 orbits 
$\tilde{\An{O}}_1 = \A(\tilde{W'})$, and $\tilde{\An{O}}_2 = \A \setminus \A(\tilde{W'}) = M_j$ and $j \in \{1,\ldots,6\} \setminus \{i\}$.
Now the intersection $W' \cap \tilde{W'}$ of these two conjugate subgroups is isomorphic to
$\crgM{4}{4}{4} \leq W$ and $\crgM{4}{4}{4}$ has two orbits $\An{O}_{21}$, $\An{O}_{22}$ on $\An{O}_2$ of size $16$ and $4$,
respectively two orbits $\tilde{\An{O}}_{21}$, $\tilde{\An{O}}_{22}$ on $\tilde{\An{O}}_2$ of size $16$ and $4$
(see Definition \ref{def:ArrG31}).
Because of the cardinalities of $\A(W')$ and $\A(\tilde{W'})$ we have $M_i \cap M_j = \An{O}_2 \cap \tilde{\An{O}}_2 \neq \emptyset$,
and $M_i \cap M_j = \An{O}_{22} = \tilde{\An{O}}_{22}$. Since the collection $M_i \cap M_j$
is stabilized by $\crgM{4}{2}{4} \geq \crgM{4}{4}{4}$, the lines orthogonal to
the hyperplanes in $M_i \cap M_j$ are the unique system of imprimitivity $\crgM{4}{2}{4}$.
Hence we get $M_i \cap M_j \cong \A(A_1^4) = \{ \ker(x_i) \mid 1 \leq i \leq 4\}$.

Now let $W' = \crgM{4}{2}{4}$. Here we also have $N_W(W') = W'$, so $\Betrag{W:W'} = 15$,
and hence again with Lemma \ref{lem:Action_W_RSA} there are $15$ distinct
subarrangements isomorphic to $\A(W') \subseteq \A$.
Since each to $W'$ conjugate reflection subgroup of $W$ has a unique system of imprimitivity consisting of the lines orthogonal to
the hyperplanes in $M_i \cap M_j$ for $i,j \in \{1,\ldots,6\}, i \neq j$ and they are distinct, the $M_i \cap M_j$ are distinct
and disjoint. 

Finally each hyperplane in $\A$ belongs to a unique intersection $M_i \cap M_j$,
so they form a partition $\F$ of $\A$. 
Since $W$ acts transitively on $\A$, and interchanges the
systems of imprimitivity corresponding to the reflection subarrangements isomorphic to $\A(\crgM{4}{2}{4})$,
it acts transitively on $\F$.
\end{proof}

The partition $\F$ in Lemma \ref{lem:Part_AG31} can be visualized in a picture, see Figure \ref{fig:Part_A_15}.

In the above proof we used some facts about the actions and orders of complex reflection (sub)groups from the book
by Lehrer and Taylor, \cite{lehrer2009unitary}, (see in particular \cite[Ch.~8, 10.5]{lehrer2009unitary}).

Furthermore it will be helpful to know the distribution of the $\A_X, X \in \LAq{\A}{2}$ with respect to the partition
given by Lemma \ref{lem:Part_AG31}:

\begin{lemma}\label{lem:Part_A_X}
Let $H \in \A$, $X \in \A^H$, and $H \in \B_{ij} := M_i\cap M_j \in \F$ for some $1 \leq i < j \leq 6$.
For $\A_X$ there are $3$ cases:
\begin{enumerate}
\item $A_X = \{H,K_1,\ldots,K_5\} \cong \A(\crgM{4}{2}{2})$ with $K_1 \in \B$, $\{K_2,K_3\} \subseteq \B_{km} = M_k\cap M_m$,
	and $\{K_4,K_5\} \subseteq \B_{pq} = M_p\cap M_q$, with $\{i,j,k,m,p,q\} = \{1,\ldots,6\}$.
\item $A_X = \{H,K_1,K_2\} \cong \A(A_2)$ with $K_1 \in \B_{ik} = M_i \cap M_k$, and $K_2 \in \B_{jK} = M_j \cap M_k$
	for some $k \in \{1,\ldots,6\} \setminus \{i,j\}$.
\item $A_X = \{H,K\} \cong \A(A_1^2)$  with $K \in \B_{km} = M_k\cap M_m$ for some $k,m \in \{1,\ldots,6\} \setminus \{i,j\}$.
\end{enumerate}
\end{lemma}

\begin{proof}
This is by explicitly writing down the partition $\F$ from Lemma \ref{lem:Part_AG31} with respect to definition \ref{def:ArrG31}
and a simple computation.
\end{proof}

The following lemma provides the next step towards a complete characterization of the \FFSAs of $\A(\crg{31})$ .

\begin{lemma}\label{lem:G29_G31_arbt_desc}
Let $\An{M} \subseteq \A := \A(\crg{31})$ be a subcollection, such that $\An{B} = \A \setminus \An{M} \cong \A(\crg{29})$.
Then for all $\An{N} \subseteq \An{M}$, $\tilde{\A}:= \A \setminus \An{N}$ is a \FFSA with exponents
$\expA{\A^{(\An{N})}}{ 1,13,17,29 - \Betrag{\An{N}} }$.
\end{lemma}
\begin{proof}

Let $\An{M} \subseteq \A$ such that $\An{B} = \A \setminus \An{M} \cong \A(\crg{29})$.
We claim that $\An{M}$ satisfies condition (\ref{eq:SubsetN_FFSA}), so with Lemma \ref{lem:SubsetN_FFSA},
$\An{B}$ is a \FFSAp. Furthermore, if $\An{M}$ satisfies condition (\ref{eq:SubsetN_FFSA}), so does every subcollection
$\An{N} \subseteq \An{M}$ and $\tilde{\A} := \A \setminus \An{N}$ is a \FFSA with exponents
$\expA{\tilde{\A}}{ 1,13,17,29 - \Betrag{\An{N}} }$.

Now let $H \in \An{M}$ be an arbitrary hyperplane in $\An{M}$ and let $X \in \A^H$. 
Then by Proposition \ref{lem:Part_A_X} there are three different cases:
\begin{eqnarray*}
\text{(1) } \Betrag{\A_X} &= 2 \text{,}& \A_X = \{ H,K \}, \\
\text{(2) } \Betrag{\A_X} &= 3 \text{,}& \A_X = \{ H,H',K \}, \\
\text{(3) } \Betrag{\A_X} &= 6 \text{,}& \A_X = \{ H,H',K_1,\ldots,K_4 \},
\end{eqnarray*}
with $H' \in \An{M}$ and $K,K_i \in \B \cong \A(\crg{29})$.
For arbitrary $H, H' \in \An{M}$ there is a hyperplane $K \in \B$ such that $H \cap H' = X \subseteq K$.
Hence $\An{M}$ satisfies condition (\ref{eq:SubsetN_FFSA}) and as mentioned before with Lemma \ref{lem:SubsetN_FFSA}
$\tilde{\A}$ is a \FFSA with exponents $\expA{\tilde{\A}}{ 1,13,17,29 - \Betrag{\An{N}} }$.
\end{proof}

The next lemma completes the characterization of the \FFSAs $\tilde{\A} \subseteq \A(\crg{31})$
and enables us to prove Proposition \ref{prop:G31_NIF}.

\begin{lemma}\label{lem:FFSA_G31}
Let $\A = \A(\crg{31})$.
A subarrangement $\A \setminus \An{N} = \tilde{\A} \subseteq \A$ is a \FFSA if and only if
\begin{enumerate}
\item $\A(\crg{29}) \subseteq \tilde{\A}$ \\
or
\item $\Betrag{\An{N}} \leq 13$ and $\An{N}$ satisfies (\ref{eq:SubsetN_FFSA}) from Lemma \ref{lem:SubsetN_FFSA}.
\end{enumerate}
In both cases the exponents of $\tilde{\A}$ are $\expA{\tilde{\A}}{1,13,17,29-\Betrag{\An{N}}}$.
\end{lemma}

\begin{proof}
Let $\tilde{\A} \subseteq \A$ be a subarrangement.
If $\tilde{\A}$ satisfies (1) then by Lemma \ref{lem:G29_G31_arbt_desc} it is a \FFSA and if $\tilde{\A}$ satisfies (2) then
by Lemma \ref{lem:SubsetN_FFSA} it is also a \FFSAp.
This gives one direction.

The other direction requires more effort. 
The main idea is to use the partion $\F$ of $\A$ from Lemma \ref{lem:Part_AG31}, the distribution of the localizations
$\A_X$ with respect to this partion given by Lemma \ref{lem:Part_A_X}, and some counting arguments.

So let $\A \setminus \N' = \tilde{\A}' \subseteq \A$ be a subarrangement such that $\A(\crg{29}) \nsubseteq \tilde{\A}'$,
$\Betrag{\N'} \geq 14$, and suppose that $\tilde{\A}'$ is a \FFSAp.
Since $\tilde{\A}'$ is a \FFSA there has to be another \FFSA say $\tilde{\A} \supseteq \tilde{\A}'$, 
$\tilde{\A} = \A \setminus \N$ such that $\Betrag{\N} = 13$.
By Lemma \ref{lem:SubsetN_FFSA} we then have $\bigcup_{X \in \LAq{\N}{2}} X \subseteq \bigcup_{H \in \tilde{\A}} H$
and $\expA{\tilde{\A}}{1,13,16,17}$.
We claim that there is no $H \in \tilde{\A}$ such that $\Betrag{\tilde{\A}^H} \in \{30,31\}$, so by Theorem \ref{thm:A_AoH_exp}
contradicting the fact that $\tilde{\A}'$ is a \FFSAp.

If $\A(\crg{29}) \subseteq \tilde{\A}$ then by Lemma \ref{lem:Part_AG31} there is an $1 \leq i \leq 6$ such that
$\N \subseteq M_i$.
With respect to renumbering the $M_i$ we may assume that $\N \subseteq M_1$.
Let $\B_{1j} = M_1 \cap M_j$, $2 \leq j \leq 6$ be the blocks of the partition of $M_1$ from
Lemma \ref{lem:Part_AG31}.
Since $\Betrag{\N} = 13$ we have $\B_{1j} \cap \N \neq \emptyset$, and there is a $k$ such that $\Betrag{\B_{1k} \cap \N} \geq 3$.
By $\tilde{\A}' \supsetneq \A(\crg{29})$, we have $H \notin M_1$.
But then, using Lemma \ref{lem:Part_A_X}, we see that $\Betrag{\tilde{\A}^H} < 30$ (because
$\N$ completely contains at least two localizations as in Lemma \ref{lem:Part_A_X}(2), and (3)), so $\tilde{\A}'$ is not free
by Theorem \ref{thm:A_AoH_exp} and in particular it is not a \FFSA contradicting our assumption.

If $\A(\crg{29}) \nsubseteq \tilde{\A}$ we claim that for such a \FFSA $\tilde{\A}$ with $\Betrag{\N} = 13$
there is a $H \in \A$, $H \in \B \in \F$ (see Lemma \ref{lem:Part_AG31}), such that 
\begin{align}\label{N_13}
\N =  \bigcup_{H' \in \B \setminus \{H\}} \A_{H \cap H'} \setminus \{H'\},
\end{align}
which enables us to describe $\tilde{\A}^K$ for each $K \in \tilde{\A}$.

So let $\tilde{\A} = \A \setminus \N$ be a \FFSA with $\A(\crg{29}) \nsubseteq \tilde{\A}$ and $\Betrag{\N} = 13$.
By Lemma \ref{lem:SubsetN_FFSA} $\N$ has to satisfy condition (\ref{eq:SubsetN_FFSA}).
Let $\F_\N := \{ \B \in \F \mid \N \cap \B \neq \emptyset \}$ be the blocks in the partition $\F$ of $\A$ containing the hyperplanes
from $\N$ and let $\B_{ab} := M_a \cap M_b \in \F$ ($a \neq b$, $a,b \in \{1,\ldots,6\}$).
First we notice that $\Betrag{\F_\N} \geq 4$ because $\Betrag{\N} = 13$.
Since $\A(\crg{29}) \nsubseteq \tilde{\A}$, by Lemma \ref{lem:Part_AG31} we have one of the following cases
\begin{enumerate}
\item there are $\B_{ij}, \B_{kl} \in \F_\N$, such that $\Betrag{\{i,j,k,l\}}=4$, 
\item there are $\B_{ij}, \B_{ik}, \B_{jk} \in \F_\N$, such that $\Betrag{\{i,j,k\}}=3$.
\end{enumerate}
But since $\Betrag{\F_\N} \geq 4$, in case (2) there is a $\B_{ab} \in \F_\N$ with $a \in \{i,k,l\}$ and $b \notin \{i,j,k\}$,
so we are again in case (1), (compare with Figure \ref{fig:Part_A_15}).
Hence (with possibly renumering the $M_i$) we have $\B_{12}, \B_{34} \in \F_\N$.
By the distribution of the simply intersecting hyperplanes in $\A$ with respect to $\F$ (Lemma \ref{lem:Part_A_X}(3))
and by condition (\ref{eq:SubsetN_FFSA}) we further have $\Betrag{\N\cap\B_{12}} \leq 2$, $\Betrag{\N\cap\B_{34}} \leq 2$
resulting in $\Betrag{\F_\N} \geq 5$.
Next, suppose for all $\B_{ab} \in \F$ we have $\{a,b\} \subseteq \{1,2,3,4\}$, so in particular $\N \subseteq \A(\crgM{4}{4}{4})$
(see Figure \ref{fig:Part_A_15}, Definition \ref{def:ArrG31} and Lemma \ref{lem:Part_AG31}).
Then because of $\Betrag{\N\cap\B_{12}} \leq 2$, $\Betrag{\N\cap\B_{34}} \leq 2$, $\Betrag{N} = 13$, and
$\Betrag{\F_\N} \geq 5$ we find $\B_{1a}, \B_{2b} \in \F_\N$, $a,b \in \{3,4\}$, such that $\Betrag{(\B_{1a} \cup \B_{2b}) \cap \N} \geq 5$.
But this contradicts condition (\ref{eq:SubsetN_FFSA}) by Lemma \ref{lem:Part_A_X}(2).
So there is a $\B_{ab} \in \F_\N$ with $\{a,b\} \nsubseteq \{1,2,3,4\}$.
Now for $\B_{ab}$ there are again two possible cases
\begin{enumerate}
\item $a=5$ and $b=6$,
\item $a \in \{1,2,3,4\}$ and $b \in \{5,6\}$.
\end{enumerate}
In the first case, by Lemma \ref{lem:Part_A_X}(3) and condition (\ref{eq:SubsetN_FFSA}), we then have $\Betrag{\N\cap\B} \leq 2$
for all $\B \in \F_\N$ so $\Betrag{\F_\N} \geq 7$.
So in this (after renumbering the $M_i$ once more) we may assume that we are in the second case.
In the second case, again by Lemma \ref{lem:Part_A_X}(3) and condition (\ref{eq:SubsetN_FFSA})
we then have $\Betrag{\B_{ij} \cap \N}\leq 2$ for $i \neq a, j \neq a$. 
We may assume that $a=1$ (the other cases are similar),
then only $\Betrag{(\B_{13}\cup\B_{14})\cap\N} \leq 4$ by Lemma \ref{lem:Part_A_X}(2) and condition (\ref{eq:SubsetN_FFSA}).
So in this case we also have $\Betrag{\F_\N} \geq 7$ and further $\Betrag{\B_{34}\cap\N} = 1$ by Lemma \ref{lem:Part_A_X}(3).

We remark that for a subarrangement $\C \subseteq \A$ with $\C \cong \A(\crgM{4}{2}{4})$ there is a $\B_{ij} \in \F$,
such that $\C = \B_{ij} \cup (\A \setminus (M_i \cup M_j)) = \B_{ij} \cup \bigcup_{a,b \in \{1,\ldots,6\} \setminus \{i,j\}} \B_{ab}$
(compare again with Figure \ref{fig:Part_A_15}, Definition \ref{def:ArrG31} and Lemma \ref{lem:Part_AG31}).
If $\N$ is of the claimed form (\ref{N_13}), by Lemma \ref{lem:Part_A_X}(1) we have $\N \subseteq \A(\crgM{4}{2}{4})$
and furthermore, since $\Betrag{\N} = 13$ and $\N$ has to satisfy condition (\ref{eq:SubsetN_FFSA}), with Lemma 
\ref{lem:Part_A_X} one easily sees, that if $\N \subseteq \A(\crgM{4}{2}{4})$, it has to be of the form (\ref{N_13}).

To finally prove the claim, we want to show that $\N \subseteq \A(\crgM{4}{2}{4})$ (for one possible realization of $\A(\crgM{4}{2}{4})$
inside $\A$ given by $\F$).

So far we have that there are $\B_{12}, \B_{34}, \B_{1b} \in \F_\N$ ($b \in \{5,6\}$).
This can be visualized in the following picture (Figure \ref{Fig_N_p}(a), compare also with Figure \ref{fig:Part_A_15}), 
where the boxes represent the partition $\F$, 
a double circle represents a hyperplane already fixed in $\N$ by the above considerations,
a solid circle a hyperplane which can not belong to $\N$ without violating condition (\ref{eq:SubsetN_FFSA}),
and a non solid circle a hyperplane which may or may not belong to $\N$.

\begin{figure}[h!]
\setlength{\unitlength}{4mm}

\begin{subfigure}[b]{0.05\textwidth}
\begin{picture}(1.7,8)
\put(1.5,4){(a)}
\end{picture}
\end{subfigure}
%3
\begin{subfigure}[b]{0.3\textwidth}
\begin{picture}(10,8)
\put(0,7.5){\line(1,0){10}}
\put(0,6){\line(1,0){10}}
\put(2,4.5){\line(1,0){8}}
\put(4,3){\line(1,0){6}}
\put(6,1.5){\line(1,0){4}}
\put(8,0){\line(1,0){2}}
\put(10,7.5){\line(0,-1){7.5}}
\put(8,7.5){\line(0,-1){7.5}}
\put(6,7.5){\line(0,-1){6}}
\put(4,7.5){\line(0,-1){4.5}}
\put(2,7.5){\line(0,-1){3}}
\put(0,7.5){\line(0,-1){1.5}}

%1n2
\put(0.5,7.1){\circle{0.6}}
\put(0.5,7.1){\circle*{0.38}}
\put(1.5,7.1){\circle{0.6}}
\put(0.5,6.4){\circle*{0.6}}
\put(1.5,6.4){\circle*{0.6}}
%1n3
\put(2.5,7.1){\circle{0.6}}
\put(3.5,7.1){\circle{0.6}}
\put(2.5,6.4){\circle{0.6}}
\put(3.5,6.4){\circle{0.6}}
%1n4
\put(4.5,7.1){\circle{0.6}}
\put(5.5,7.1){\circle{0.6}}
\put(4.5,6.4){\circle{0.6}}
\put(5.5,6.4){\circle{0.6}}
%1n5
\put(6.5,7.1){\circle*{0.6}}
\put(7.5,7.1){\circle*{0.6}}
\put(6.5,6.4){\circle{0.6}}
\put(7.5,6.4){\circle{0.6}}
%1n6
\put(8.5,7.1){\circle{0.6}}
\put(8.5,7.1){\circle*{0.38}}
\put(9.5,7.1){\circle{0.6}}
\put(8.5,6.4){\circle*{0.6}}
\put(9.5,6.4){\circle*{0.6}}

%2n3
\put(2.5,5.6){\circle{0.6}}
\put(3.5,5.6){\circle{0.6}}
\put(2.5,4.9){\circle*{0.6}}
\put(3.5,4.9){\circle*{0.6}}
%2n4
\put(4.5,5.6){\circle{0.6}}
\put(5.5,5.6){\circle{0.6}}
\put(4.5,4.9){\circle*{0.6}}
\put(5.5,4.9){\circle*{0.6}}
%2n5
\put(6.5,5.6){\circle{0.6}}
\put(7.5,5.6){\circle{0.6}}
\put(6.5,4.9){\circle*{0.6}}
\put(7.5,4.9){\circle*{0.6}}
%2n6
\put(8.5,5.6){\circle*{0.6}}
\put(9.5,5.6){\circle*{0.6}}
\put(8.5,4.9){\circle{0.6}}
\put(9.5,4.9){\circle{0.6}}

%3n4
\put(4.5,4.1){\circle*{0.6}}
\put(5.5,4.1){\circle{0.6}}
\put(5.5,4.1){\circle*{0.38}}
\put(4.5,3.4){\circle*{0.6}}
\put(5.5,3.4){\circle*{0.6}}
%3n5
\put(6.5,4.1){\circle*{0.6}}
\put(7.5,4.1){\circle{0.6}}
\put(6.5,3.4){\circle*{0.6}}
\put(7.5,3.4){\circle*{0.6}}
%3n6
\put(8.5,4.1){\circle{0.6}}
\put(9.5,4.1){\circle{0.6}}
\put(8.5,3.4){\circle*{0.6}}
\put(9.5,3.4){\circle*{0.6}}

%4n5
\put(6.5,2.6){\circle*{0.6}}
\put(7.5,2.6){\circle{0.6}}
\put(6.5,1.9){\circle*{0.6}}
\put(7.5,1.9){\circle*{0.6}}
%4n6
\put(8.5,2.6){\circle{0.6}}
\put(9.5,2.6){\circle{0.6}}
\put(8.5,1.9){\circle*{0.6}}
\put(9.5,1.9){\circle*{0.6}}

%5n6
\put(8.5,1.1){\circle{0.6}}
\put(9.5,1.1){\circle{0.6}}
\put(8.5,0.4){\circle*{0.6}}
\put(9.5,0.4){\circle*{0.6}}

\end{picture}
\end{subfigure}
%%%
\begin{subfigure}[b]{0.05\textwidth}
\begin{picture}(1.7,8)
\put(1.5,4){(b)}
\end{picture}
\end{subfigure}
%4
\begin{subfigure}[b]{0.3\textwidth}
\begin{picture}(10,8)
\put(0,7.5){\line(1,0){10}}
\put(0,6){\line(1,0){10}}
\put(2,4.5){\line(1,0){8}}
\put(4,3){\line(1,0){6}}
\put(6,1.5){\line(1,0){4}}
\put(8,0){\line(1,0){2}}
\put(10,7.5){\line(0,-1){7.5}}
\put(8,7.5){\line(0,-1){7.5}}
\put(6,7.5){\line(0,-1){6}}
\put(4,7.5){\line(0,-1){4.5}}
\put(2,7.5){\line(0,-1){3}}
\put(0,7.5){\line(0,-1){1.5}}

%1n2
\put(0.5,7.1){\circle{0.6}}
\put(0.5,7.1){\circle*{0.38}}
\put(1.5,7.1){\circle*{0.6}}
\put(0.5,6.4){\circle*{0.6}}
\put(1.5,6.4){\circle*{0.6}}
%1n3
\put(2.5,7.1){\circle{0.6}}
\put(3.5,7.1){\circle*{0.6}}
\put(2.5,6.4){\circle{0.6}}
\put(3.5,6.4){\circle{0.6}}
%1n4
\put(4.5,7.1){\circle{0.6}}
\put(5.5,7.1){\circle*{0.6}}
\put(4.5,6.4){\circle*{0.6}}
\put(5.5,6.4){\circle{0.6}}
%1n5
\put(6.5,7.1){\circle*{0.6}}
\put(7.5,7.1){\circle*{0.6}}
\put(6.5,6.4){\circle*{0.6}}
\put(7.5,6.4){\circle{0.6}}
%1n6
\put(8.5,7.1){\circle{0.6}}
\put(8.5,7.1){\circle*{0.38}}
\put(9.5,7.1){\circle{0.6}}
\put(8.5,6.4){\circle*{0.6}}
\put(9.5,6.4){\circle*{0.6}}

%2n3
\put(2.5,5.6){\circle{0.6}}
\put(3.5,5.6){\circle{0.6}}
\put(2.5,4.9){\circle*{0.6}}
\put(3.5,4.9){\circle*{0.6}}
%2n4
\put(4.5,5.6){\circle{0.6}}
\put(5.5,5.6){\circle*{0.6}}
\put(4.5,4.9){\circle*{0.6}}
\put(5.5,4.9){\circle*{0.6}}
%2n5
\put(6.5,5.6){\circle{0.6}}
\put(7.5,5.6){\circle*{0.6}}
\put(6.5,4.9){\circle*{0.6}}
\put(7.5,4.9){\circle*{0.6}}
%2n6
\put(8.5,5.6){\circle*{0.6}}
\put(9.5,5.6){\circle*{0.6}}
\put(8.5,4.9){\circle{0.6}}
\put(9.5,4.9){\circle{0.6}}

%3n4
\put(4.5,4.1){\circle*{0.6}}
\put(5.5,4.1){\circle{0.6}}
\put(5.5,4.1){\circle*{0.38}}
\put(4.5,3.4){\circle*{0.6}}
\put(5.5,3.4){\circle*{0.6}}
%3n5
\put(6.5,4.1){\circle*{0.6}}
\put(7.5,4.1){\circle{0.6}}
\put(6.5,3.4){\circle*{0.6}}
\put(7.5,3.4){\circle*{0.6}}
%3n6
\put(8.5,4.1){\circle{0.6}}
\put(9.5,4.1){\circle{0.6}}
\put(9.5,4.1){\circle*{0.38}}
\put(8.5,3.4){\circle*{0.6}}
\put(9.5,3.4){\circle*{0.6}}

%4n5
\put(6.5,2.6){\circle*{0.6}}
\put(7.5,2.6){\circle{0.6}}
\put(6.5,1.9){\circle*{0.6}}
\put(7.5,1.9){\circle*{0.6}}
%4n6
\put(8.5,2.6){\circle{0.6}}
\put(9.5,2.6){\circle{0.6}}
\put(8.5,1.9){\circle*{0.6}}
\put(9.5,1.9){\circle*{0.6}}

%5n6
\put(8.5,1.1){\circle{0.6}}
\put(9.5,1.1){\circle{0.6}}
\put(8.5,0.4){\circle*{0.6}}
\put(9.5,0.4){\circle*{0.6}}

\end{picture}
\end{subfigure}

\caption{Possible choices for hyperplanes in $\N$.}%
\label{Fig_N_p}
\end{figure}
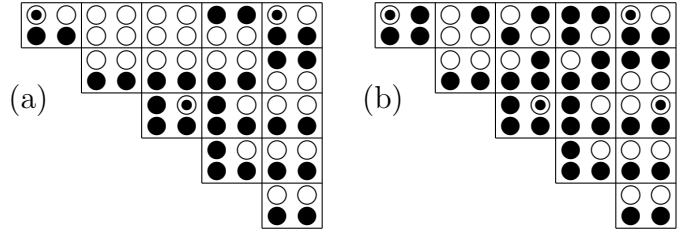

Suppose that there is a $\B_{cd} \in \F_\N$ such that $\{c,d\} \cap \{3,4\} \neq \emptyset$.
This is the case if and only if $\N \subseteq \A(\crgM{4}{2}{4})$ by our remark before.

Then the hyperplanes left to be chosen for $\N$ reduce considerably (see Figure \ref{Fig_N_p}(b)).

If we proceed in this manner using the same arguments as above we arrive at a contradiction to $\Betrag{\N} = 13$,
condition (\ref{eq:SubsetN_FFSA}), and Lemma \ref{lem:Part_A_X}.

To finish the proof, let $\tilde{\A} = \A \setminus \N$ for an $\N$ of the form (\ref{N_13}).
Then by Lemma \ref{lem:Part_A_X}(3) and the distribution of the $H \in \tilde{\A}$ 
with respect to $\F$ we have $\Betrag{\tilde{\A}^H} \leq 29$ since for $H$ there are at least two hyperplanes in $\N$
simply intersecting $H$ and we are done.
\end{proof}

\begin{example}
We illustrate the change of the set of hyperplanes which can be added to $\An{N}$
along a free filtration from $\A$ to $\A \setminus \An{N} = \tilde{\A}$ with $\Betrag{\tilde{\A}} =47$,
$\A(\crg{29}) \nsubseteq \tilde{\A}$, by the following sequence of pictures (Figure \ref{Fig_Rem_M31}).
Each circle represents a hyperplane in the \FFSA $\A_i$, a solid circle represents a hyperplane
which we can not add to $\An{N}$ without violating condition (\ref{eq:SubsetN_FFSA}) 
from Lemma \ref{lem:SubsetN_FFSA}. 
A non-solid circle represents a hyperplane, which can be added to $\An{N}$, such that  (\ref{eq:SubsetN_FFSA}) ist
still satisfied.
The different boxes represent the partition $\F$
of $\A$ into subsets of $4$ hyperplanes given by Lemma \ref{lem:FFSA_G31}:

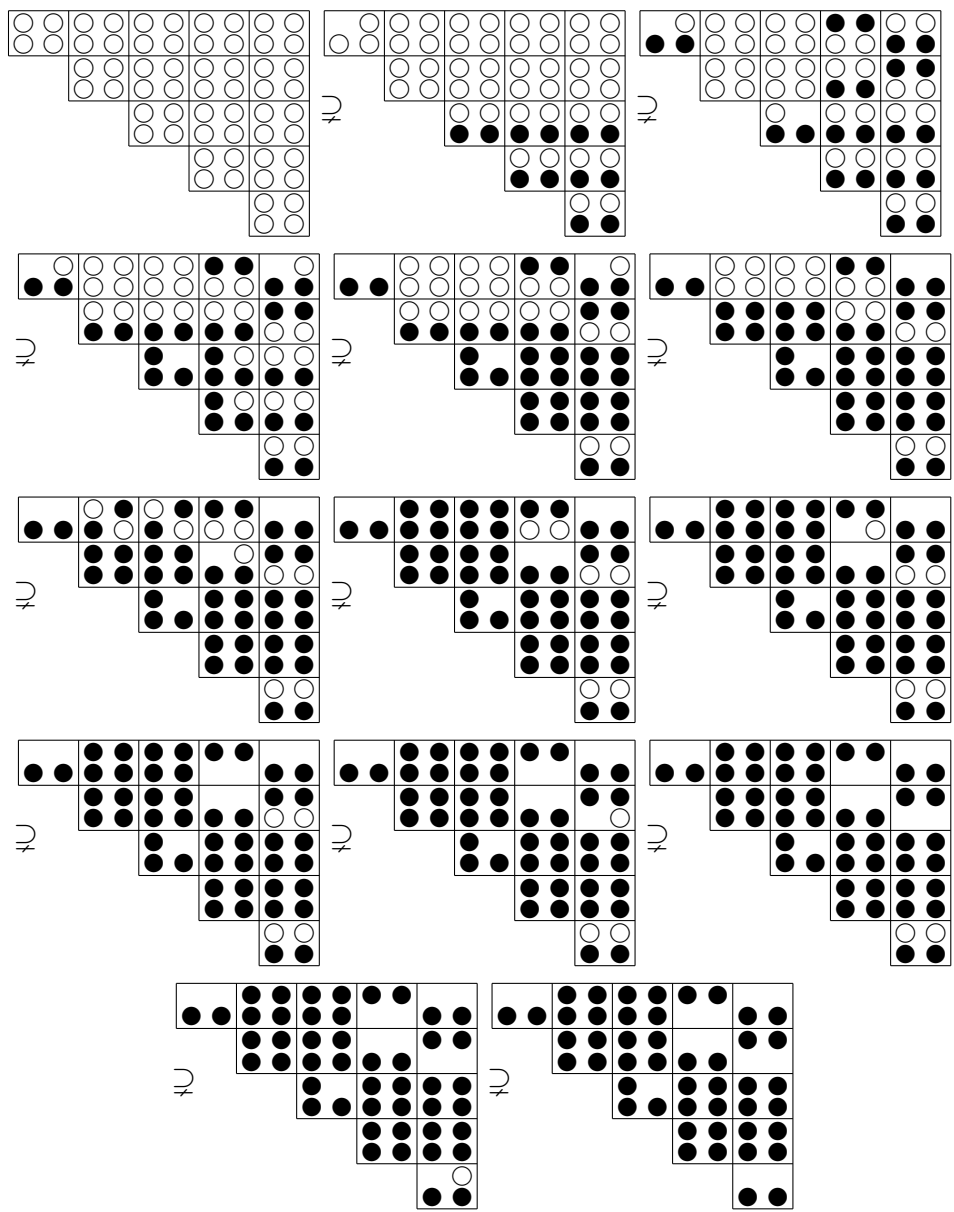
\begin{figure}[h!]
\setlength{\unitlength}{4mm}

\begin{subfigure}[b]{0.3\textwidth}
\begin{picture}(10,8)
\put(0,7.5){\line(1,0){10}}
\put(0,6){\line(1,0){10}}
\put(2,4.5){\line(1,0){8}}
\put(4,3){\line(1,0){6}}
\put(6,1.5){\line(1,0){4}}
\put(8,0){\line(1,0){2}}
\put(10,7.5){\line(0,-1){7.5}}
\put(8,7.5){\line(0,-1){7.5}}
\put(6,7.5){\line(0,-1){6}}
\put(4,7.5){\line(0,-1){4.5}}
\put(2,7.5){\line(0,-1){3}}
\put(0,7.5){\line(0,-1){1.5}}

%1n2
\put(0.5,7.1){\circle{0.6}}
\put(1.5,7.1){\circle{0.6}}
\put(0.5,6.4){\circle{0.6}}
\put(1.5,6.4){\circle{0.6}}
%1n3
\put(2.5,7.1){\circle{0.6}}
\put(3.5,7.1){\circle{0.6}}
\put(2.5,6.4){\circle{0.6}}
\put(3.5,6.4){\circle{0.6}}
%1n4
\put(4.5,7.1){\circle{0.6}}
\put(5.5,7.1){\circle{0.6}}
\put(4.5,6.4){\circle{0.6}}
\put(5.5,6.4){\circle{0.6}}
%1n5
\put(6.5,7.1){\circle{0.6}}
\put(7.5,7.1){\circle{0.6}}
\put(6.5,6.4){\circle{0.6}}
\put(7.5,6.4){\circle{0.6}}
%1n6
\put(8.5,7.1){\circle{0.6}}
\put(9.5,7.1){\circle{0.6}}
\put(8.5,6.4){\circle{0.6}}
\put(9.5,6.4){\circle{0.6}}

%2n3
\put(2.5,5.6){\circle{0.6}}
\put(3.5,5.6){\circle{0.6}}
\put(2.5,4.9){\circle{0.6}}
\put(3.5,4.9){\circle{0.6}}
%2n4
\put(4.5,5.6){\circle{0.6}}
\put(5.5,5.6){\circle{0.6}}
\put(4.5,4.9){\circle{0.6}}
\put(5.5,4.9){\circle{0.6}}
%2n5
\put(6.5,5.6){\circle{0.6}}
\put(7.5,5.6){\circle{0.6}}
\put(6.5,4.9){\circle{0.6}}
\put(7.5,4.9){\circle{0.6}}
%2n6
\put(8.5,5.6){\circle{0.6}}
\put(9.5,5.6){\circle{0.6}}
\put(8.5,4.9){\circle{0.6}}
\put(9.5,4.9){\circle{0.6}}

%3n4
\put(4.5,4.1){\circle{0.6}}
\put(5.5,4.1){\circle{0.6}}
\put(4.5,3.4){\circle{0.6}}
\put(5.5,3.4){\circle{0.6}}
%3n5
\put(6.5,4.1){\circle{0.6}}
\put(7.5,4.1){\circle{0.6}}
\put(6.5,3.4){\circle{0.6}}
\put(7.5,3.4){\circle{0.6}}
%3n6
\put(8.5,4.1){\circle{0.6}}
\put(9.5,4.1){\circle{0.6}}
\put(8.5,3.4){\circle{0.6}}
\put(9.5,3.4){\circle{0.6}}

%4n5
\put(6.5,2.6){\circle{0.6}}
\put(7.5,2.6){\circle{0.6}}
\put(6.5,1.9){\circle{0.6}}
\put(7.5,1.9){\circle{0.6}}
%4n6
\put(8.5,2.6){\circle{0.6}}
\put(9.5,2.6){\circle{0.6}}
\put(8.5,1.9){\circle{0.6}}
\put(9.5,1.9){\circle{0.6}}

%5n6
\put(8.5,1.1){\circle{0.6}}
\put(9.5,1.1){\circle{0.6}}
\put(8.5,0.4){\circle{0.6}}
\put(9.5,0.4){\circle{0.6}}

\end{picture}
\end{subfigure}
%%%
\begin{subfigure}[b]{0.01\textwidth}
\begin{picture}(0.5,8)
\put(0.5,4){$\supsetneq$}
\end{picture}
\end{subfigure}
%1
\begin{subfigure}[b]{0.3\textwidth}
\begin{picture}(10,8)
\put(0,7.5){\line(1,0){10}}
\put(0,6){\line(1,0){10}}
\put(2,4.5){\line(1,0){8}}
\put(4,3){\line(1,0){6}}
\put(6,1.5){\line(1,0){4}}
\put(8,0){\line(1,0){2}}
\put(10,7.5){\line(0,-1){7.5}}
\put(8,7.5){\line(0,-1){7.5}}
\put(6,7.5){\line(0,-1){6}}
\put(4,7.5){\line(0,-1){4.5}}
\put(2,7.5){\line(0,-1){3}}
\put(0,7.5){\line(0,-1){1.5}}

%1n2
%\put(0.5,7.1){\circle{0.6}}
%\put(0.5,7.1){\circle*{0.38}}
\put(1.5,7.1){\circle{0.6}}
\put(0.5,6.4){\circle{0.6}}
\put(1.5,6.4){\circle{0.6}}
%1n3
\put(2.5,7.1){\circle{0.6}}
\put(3.5,7.1){\circle{0.6}}
\put(2.5,6.4){\circle{0.6}}
\put(3.5,6.4){\circle{0.6}}
%1n4
\put(4.5,7.1){\circle{0.6}}
\put(5.5,7.1){\circle{0.6}}
\put(4.5,6.4){\circle{0.6}}
\put(5.5,6.4){\circle{0.6}}
%1n5
\put(6.5,7.1){\circle{0.6}}
\put(7.5,7.1){\circle{0.6}}
\put(6.5,6.4){\circle{0.6}}
\put(7.5,6.4){\circle{0.6}}
%1n6
\put(8.5,7.1){\circle{0.6}}
\put(9.5,7.1){\circle{0.6}}
\put(8.5,6.4){\circle{0.6}}
\put(9.5,6.4){\circle{0.6}}

%2n3
\put(2.5,5.6){\circle{0.6}}
\put(3.5,5.6){\circle{0.6}}
\put(2.5,4.9){\circle{0.6}}
\put(3.5,4.9){\circle{0.6}}
%2n4
\put(4.5,5.6){\circle{0.6}}
\put(5.5,5.6){\circle{0.6}}
\put(4.5,4.9){\circle{0.6}}
\put(5.5,4.9){\circle{0.6}}
%2n5
\put(6.5,5.6){\circle{0.6}}
\put(7.5,5.6){\circle{0.6}}
\put(6.5,4.9){\circle{0.6}}
\put(7.5,4.9){\circle{0.6}}
%2n6
\put(8.5,5.6){\circle{0.6}}
\put(9.5,5.6){\circle{0.6}}
\put(8.5,4.9){\circle{0.6}}
\put(9.5,4.9){\circle{0.6}}

%3n4
\put(4.5,4.1){\circle{0.6}}
\put(5.5,4.1){\circle{0.6}}
\put(4.5,3.4){\circle*{0.6}}
\put(5.5,3.4){\circle*{0.6}}
%3n5
\put(6.5,4.1){\circle{0.6}}
\put(7.5,4.1){\circle{0.6}}
\put(6.5,3.4){\circle*{0.6}}
\put(7.5,3.4){\circle*{0.6}}
%3n6
\put(8.5,4.1){\circle{0.6}}
\put(9.5,4.1){\circle{0.6}}
\put(8.5,3.4){\circle*{0.6}}
\put(9.5,3.4){\circle*{0.6}}

%4n5
\put(6.5,2.6){\circle{0.6}}
\put(7.5,2.6){\circle{0.6}}
\put(6.5,1.9){\circle*{0.6}}
\put(7.5,1.9){\circle*{0.6}}
%4n6
\put(8.5,2.6){\circle{0.6}}
\put(9.5,2.6){\circle{0.6}}
\put(8.5,1.9){\circle*{0.6}}
\put(9.5,1.9){\circle*{0.6}}

%5n6
\put(8.5,1.1){\circle{0.6}}
\put(9.5,1.1){\circle{0.6}}
\put(8.5,0.4){\circle*{0.6}}
\put(9.5,0.4){\circle*{0.6}}

\end{picture}
\end{subfigure}
%%%
\begin{subfigure}[b]{0.01\textwidth}
\begin{picture}(0.5,8)
\put(0.5,4){$\supsetneq$}
\end{picture}
\end{subfigure}
%2
\begin{subfigure}[b]{0.3\textwidth}
\begin{picture}(10,8)
\put(0,7.5){\line(1,0){10}}
\put(0,6){\line(1,0){10}}
\put(2,4.5){\line(1,0){8}}
\put(4,3){\line(1,0){6}}
\put(6,1.5){\line(1,0){4}}
\put(8,0){\line(1,0){2}}
\put(10,7.5){\line(0,-1){7.5}}
\put(8,7.5){\line(0,-1){7.5}}
\put(6,7.5){\line(0,-1){6}}
\put(4,7.5){\line(0,-1){4.5}}
\put(2,7.5){\line(0,-1){3}}
\put(0,7.5){\line(0,-1){1.5}}

%1n2
%\put(0.5,7.1){\circle{0.6}}
%\put(0.5,7.1){\circle*{0.38}}
\put(1.5,7.1){\circle{0.6}}
\put(0.5,6.4){\circle*{0.6}}
\put(1.5,6.4){\circle*{0.6}}
%1n3
\put(2.5,7.1){\circle{0.6}}
\put(3.5,7.1){\circle{0.6}}
\put(2.5,6.4){\circle{0.6}}
\put(3.5,6.4){\circle{0.6}}
%1n4
\put(4.5,7.1){\circle{0.6}}
\put(5.5,7.1){\circle{0.6}}
\put(4.5,6.4){\circle{0.6}}
\put(5.5,6.4){\circle{0.6}}
%1n5
\put(6.5,7.1){\circle*{0.6}}
\put(7.5,7.1){\circle*{0.6}}
\put(6.5,6.4){\circle{0.6}}
\put(7.5,6.4){\circle{0.6}}
%1n6
\put(8.5,7.1){\circle{0.6}}
\put(9.5,7.1){\circle{0.6}}
\put(8.5,6.4){\circle*{0.6}}
\put(9.5,6.4){\circle*{0.6}}

%2n3
\put(2.5,5.6){\circle{0.6}}
\put(3.5,5.6){\circle{0.6}}
\put(2.5,4.9){\circle{0.6}}
\put(3.5,4.9){\circle{0.6}}
%2n4
\put(4.5,5.6){\circle{0.6}}
\put(5.5,5.6){\circle{0.6}}
\put(4.5,4.9){\circle{0.6}}
\put(5.5,4.9){\circle{0.6}}
%2n5
\put(6.5,5.6){\circle{0.6}}
\put(7.5,5.6){\circle{0.6}}
\put(6.5,4.9){\circle*{0.6}}
\put(7.5,4.9){\circle*{0.6}}
%2n6
\put(8.5,5.6){\circle*{0.6}}
\put(9.5,5.6){\circle*{0.6}}
\put(8.5,4.9){\circle{0.6}}
\put(9.5,4.9){\circle{0.6}}

%3n4
\put(4.5,4.1){\circle{0.6}}
%\put(5.5,4.1){\circle{0.6}}
%\put(5.5,4.1){\circle*{0.38}}
\put(4.5,3.4){\circle*{0.6}}
\put(5.5,3.4){\circle*{0.6}}
%3n5
\put(6.5,4.1){\circle{0.6}}
\put(7.5,4.1){\circle{0.6}}
\put(6.5,3.4){\circle*{0.6}}
\put(7.5,3.4){\circle*{0.6}}
%3n6
\put(8.5,4.1){\circle{0.6}}
\put(9.5,4.1){\circle{0.6}}
\put(8.5,3.4){\circle*{0.6}}
\put(9.5,3.4){\circle*{0.6}}

%4n5
\put(6.5,2.6){\circle{0.6}}
\put(7.5,2.6){\circle{0.6}}
\put(6.5,1.9){\circle*{0.6}}
\put(7.5,1.9){\circle*{0.6}}
%4n6
\put(8.5,2.6){\circle{0.6}}
\put(9.5,2.6){\circle{0.6}}
\put(8.5,1.9){\circle*{0.6}}
\put(9.5,1.9){\circle*{0.6}}

%5n6
\put(8.5,1.1){\circle{0.6}}
\put(9.5,1.1){\circle{0.6}}
\put(8.5,0.4){\circle*{0.6}}
\put(9.5,0.4){\circle*{0.6}}

\end{picture}
\end{subfigure}
%%%
\\
\begin{subfigure}[b]{0.01\textwidth}
\begin{picture}(0.5,8)
\put(0.5,4){$\supsetneq$}
\end{picture}
\end{subfigure}
%3
\begin{subfigure}[b]{0.3\textwidth}
\begin{picture}(10,8)
\put(0,7.5){\line(1,0){10}}
\put(0,6){\line(1,0){10}}
\put(2,4.5){\line(1,0){8}}
\put(4,3){\line(1,0){6}}
\put(6,1.5){\line(1,0){4}}
\put(8,0){\line(1,0){2}}
\put(10,7.5){\line(0,-1){7.5}}
\put(8,7.5){\line(0,-1){7.5}}
\put(6,7.5){\line(0,-1){6}}
\put(4,7.5){\line(0,-1){4.5}}
\put(2,7.5){\line(0,-1){3}}
\put(0,7.5){\line(0,-1){1.5}}

%1n2
%\put(0.5,7.1){\circle{0.6}}
%\put(0.5,7.1){\circle*{0.38}}
\put(1.5,7.1){\circle{0.6}}
\put(0.5,6.4){\circle*{0.6}}
\put(1.5,6.4){\circle*{0.6}}
%1n3
\put(2.5,7.1){\circle{0.6}}
\put(3.5,7.1){\circle{0.6}}
\put(2.5,6.4){\circle{0.6}}
\put(3.5,6.4){\circle{0.6}}
%1n4
\put(4.5,7.1){\circle{0.6}}
\put(5.5,7.1){\circle{0.6}}
\put(4.5,6.4){\circle{0.6}}
\put(5.5,6.4){\circle{0.6}}
%1n5
\put(6.5,7.1){\circle*{0.6}}
\put(7.5,7.1){\circle*{0.6}}
\put(6.5,6.4){\circle{0.6}}
\put(7.5,6.4){\circle{0.6}}
%1n6
%\put(8.5,7.1){\circle{0.6}}
%\put(8.5,7.1){\circle*{0.38}}
\put(9.5,7.1){\circle{0.6}}
\put(8.5,6.4){\circle*{0.6}}
\put(9.5,6.4){\circle*{0.6}}

%2n3
\put(2.5,5.6){\circle{0.6}}
\put(3.5,5.6){\circle{0.6}}
\put(2.5,4.9){\circle*{0.6}}
\put(3.5,4.9){\circle*{0.6}}
%2n4
\put(4.5,5.6){\circle{0.6}}
\put(5.5,5.6){\circle{0.6}}
\put(4.5,4.9){\circle*{0.6}}
\put(5.5,4.9){\circle*{0.6}}
%2n5
\put(6.5,5.6){\circle{0.6}}
\put(7.5,5.6){\circle{0.6}}
\put(6.5,4.9){\circle*{0.6}}
\put(7.5,4.9){\circle*{0.6}}
%2n6
\put(8.5,5.6){\circle*{0.6}}
\put(9.5,5.6){\circle*{0.6}}
\put(8.5,4.9){\circle{0.6}}
\put(9.5,4.9){\circle{0.6}}

%3n4
\put(4.5,4.1){\circle*{0.6}}
%\put(5.5,4.1){\circle{0.6}}
%\put(5.5,4.1){\circle*{0.38}}
\put(4.5,3.4){\circle*{0.6}}
\put(5.5,3.4){\circle*{0.6}}
%3n5
\put(6.5,4.1){\circle*{0.6}}
\put(7.5,4.1){\circle{0.6}}
\put(6.5,3.4){\circle*{0.6}}
\put(7.5,3.4){\circle*{0.6}}
%3n6
\put(8.5,4.1){\circle{0.6}}
\put(9.5,4.1){\circle{0.6}}
\put(8.5,3.4){\circle*{0.6}}
\put(9.5,3.4){\circle*{0.6}}

%4n5
\put(6.5,2.6){\circle*{0.6}}
\put(7.5,2.6){\circle{0.6}}
\put(6.5,1.9){\circle*{0.6}}
\put(7.5,1.9){\circle*{0.6}}
%4n6
\put(8.5,2.6){\circle{0.6}}
\put(9.5,2.6){\circle{0.6}}
\put(8.5,1.9){\circle*{0.6}}
\put(9.5,1.9){\circle*{0.6}}

%5n6
\put(8.5,1.1){\circle{0.6}}
\put(9.5,1.1){\circle{0.6}}
\put(8.5,0.4){\circle*{0.6}}
\put(9.5,0.4){\circle*{0.6}}

\end{picture}
\end{subfigure}
%%%
\begin{subfigure}[b]{0.01\textwidth}
\begin{picture}(0.5,8)
\put(0.5,4){$\supsetneq$}
\end{picture}
\end{subfigure}
%4
\begin{subfigure}[b]{0.3\textwidth}
\begin{picture}(10,8)
\put(0,7.5){\line(1,0){10}}
\put(0,6){\line(1,0){10}}
\put(2,4.5){\line(1,0){8}}
\put(4,3){\line(1,0){6}}
\put(6,1.5){\line(1,0){4}}
\put(8,0){\line(1,0){2}}
\put(10,7.5){\line(0,-1){7.5}}
\put(8,7.5){\line(0,-1){7.5}}
\put(6,7.5){\line(0,-1){6}}
\put(4,7.5){\line(0,-1){4.5}}
\put(2,7.5){\line(0,-1){3}}
\put(0,7.5){\line(0,-1){1.5}}

%1n2
%\put(0.5,7.1){\circle{0.6}}
%\put(0.5,7.1){\circle*{0.38}}
%\put(1.5,7.1){\circle{0.6}}
%\put(1.5,7.1){\circle*{0.38}}
\put(0.5,6.4){\circle*{0.6}}
\put(1.5,6.4){\circle*{0.6}}
%1n3
\put(2.5,7.1){\circle{0.6}}
\put(3.5,7.1){\circle{0.6}}
\put(2.5,6.4){\circle{0.6}}
\put(3.5,6.4){\circle{0.6}}
%1n4
\put(4.5,7.1){\circle{0.6}}
\put(5.5,7.1){\circle{0.6}}
\put(4.5,6.4){\circle{0.6}}
\put(5.5,6.4){\circle{0.6}}
%1n5
\put(6.5,7.1){\circle*{0.6}}
\put(7.5,7.1){\circle*{0.6}}
\put(6.5,6.4){\circle{0.6}}
\put(7.5,6.4){\circle{0.6}}
%1n6
%\put(8.5,7.1){\circle{0.6}}
%\put(8.5,7.1){\circle*{0.38}}
\put(9.5,7.1){\circle{0.6}}
\put(8.5,6.4){\circle*{0.6}}
\put(9.5,6.4){\circle*{0.6}}

%2n3
\put(2.5,5.6){\circle{0.6}}
\put(3.5,5.6){\circle{0.6}}
\put(2.5,4.9){\circle*{0.6}}
\put(3.5,4.9){\circle*{0.6}}
%2n4
\put(4.5,5.6){\circle{0.6}}
\put(5.5,5.6){\circle{0.6}}
\put(4.5,4.9){\circle*{0.6}}
\put(5.5,4.9){\circle*{0.6}}
%2n5
\put(6.5,5.6){\circle{0.6}}
\put(7.5,5.6){\circle{0.6}}
\put(6.5,4.9){\circle*{0.6}}
\put(7.5,4.9){\circle*{0.6}}
%2n6
\put(8.5,5.6){\circle*{0.6}}
\put(9.5,5.6){\circle*{0.6}}
\put(8.5,4.9){\circle{0.6}}
\put(9.5,4.9){\circle{0.6}}

%3n4
\put(4.5,4.1){\circle*{0.6}}
%\put(5.5,4.1){\circle{0.6}}
%\put(5.5,4.1){\circle*{0.38}}
\put(4.5,3.4){\circle*{0.6}}
\put(5.5,3.4){\circle*{0.6}}
%3n5
\put(6.5,4.1){\circle*{0.6}}
\put(7.5,4.1){\circle*{0.6}}
\put(6.5,3.4){\circle*{0.6}}
\put(7.5,3.4){\circle*{0.6}}
%3n6
\put(8.5,4.1){\circle*{0.6}}
\put(9.5,4.1){\circle*{0.6}}
\put(8.5,3.4){\circle*{0.6}}
\put(9.5,3.4){\circle*{0.6}}

%4n5
\put(6.5,2.6){\circle*{0.6}}
\put(7.5,2.6){\circle*{0.6}}
\put(6.5,1.9){\circle*{0.6}}
\put(7.5,1.9){\circle*{0.6}}
%4n6
\put(8.5,2.6){\circle*{0.6}}
\put(9.5,2.6){\circle*{0.6}}
\put(8.5,1.9){\circle*{0.6}}
\put(9.5,1.9){\circle*{0.6}}

%5n6
\put(8.5,1.1){\circle{0.6}}
\put(9.5,1.1){\circle{0.6}}
\put(8.5,0.4){\circle*{0.6}}
\put(9.5,0.4){\circle*{0.6}}

\end{picture}
\end{subfigure}
%%%
\begin{subfigure}[b]{0.01\textwidth}
\begin{picture}(0.5,8)
\put(0.5,4){$\supsetneq$}
\end{picture}
\end{subfigure}
%5
\begin{subfigure}[b]{0.3\textwidth}
\begin{picture}(10,8)
\put(0,7.5){\line(1,0){10}}
\put(0,6){\line(1,0){10}}
\put(2,4.5){\line(1,0){8}}
\put(4,3){\line(1,0){6}}
\put(6,1.5){\line(1,0){4}}
\put(8,0){\line(1,0){2}}
\put(10,7.5){\line(0,-1){7.5}}
\put(8,7.5){\line(0,-1){7.5}}
\put(6,7.5){\line(0,-1){6}}
\put(4,7.5){\line(0,-1){4.5}}
\put(2,7.5){\line(0,-1){3}}
\put(0,7.5){\line(0,-1){1.5}}

%1n2
%\put(0.5,7.1){\circle{0.6}}
%\put(0.5,7.1){\circle*{0.38}}
%\put(1.5,7.1){\circle{0.6}}
%\put(1.5,7.1){\circle*{0.38}}
\put(0.5,6.4){\circle*{0.6}}
\put(1.5,6.4){\circle*{0.6}}
%1n3
\put(2.5,7.1){\circle{0.6}}
\put(3.5,7.1){\circle{0.6}}
\put(2.5,6.4){\circle{0.6}}
\put(3.5,6.4){\circle{0.6}}
%1n4
\put(4.5,7.1){\circle{0.6}}
\put(5.5,7.1){\circle{0.6}}
\put(4.5,6.4){\circle{0.6}}
\put(5.5,6.4){\circle{0.6}}
%1n5
\put(6.5,7.1){\circle*{0.6}}
\put(7.5,7.1){\circle*{0.6}}
\put(6.5,6.4){\circle{0.6}}
\put(7.5,6.4){\circle{0.6}}
%1n6
%\put(8.5,7.1){\circle{0.6}}
%\put(8.5,7.1){\circle*{0.38}}
%\put(9.5,7.1){\circle{0.6}}
%\put(9.5,7.1){\circle*{0.38}}
\put(8.5,6.4){\circle*{0.6}}
\put(9.5,6.4){\circle*{0.6}}

%2n3
\put(2.5,5.6){\circle*{0.6}}
\put(3.5,5.6){\circle*{0.6}}
\put(2.5,4.9){\circle*{0.6}}
\put(3.5,4.9){\circle*{0.6}}
%2n4
\put(4.5,5.6){\circle*{0.6}}
\put(5.5,5.6){\circle*{0.6}}
\put(4.5,4.9){\circle*{0.6}}
\put(5.5,4.9){\circle*{0.6}}
%2n5
\put(6.5,5.6){\circle{0.6}}
\put(7.5,5.6){\circle{0.6}}
\put(6.5,4.9){\circle*{0.6}}
\put(7.5,4.9){\circle*{0.6}}
%2n6
\put(8.5,5.6){\circle*{0.6}}
\put(9.5,5.6){\circle*{0.6}}
\put(8.5,4.9){\circle{0.6}}
\put(9.5,4.9){\circle{0.6}}

%3n4
\put(4.5,4.1){\circle*{0.6}}
%\put(5.5,4.1){\circle{0.6}}
%\put(5.5,4.1){\circle*{0.38}}
\put(4.5,3.4){\circle*{0.6}}
\put(5.5,3.4){\circle*{0.6}}
%3n5
\put(6.5,4.1){\circle*{0.6}}
\put(7.5,4.1){\circle*{0.6}}
\put(6.5,3.4){\circle*{0.6}}
\put(7.5,3.4){\circle*{0.6}}
%3n6
\put(8.5,4.1){\circle*{0.6}}
\put(9.5,4.1){\circle*{0.6}}
\put(8.5,3.4){\circle*{0.6}}
\put(9.5,3.4){\circle*{0.6}}

%4n5
\put(6.5,2.6){\circle*{0.6}}
\put(7.5,2.6){\circle*{0.6}}
\put(6.5,1.9){\circle*{0.6}}
\put(7.5,1.9){\circle*{0.6}}
%4n6
\put(8.5,2.6){\circle*{0.6}}
\put(9.5,2.6){\circle*{0.6}}
\put(8.5,1.9){\circle*{0.6}}
\put(9.5,1.9){\circle*{0.6}}

%5n6
\put(8.5,1.1){\circle{0.6}}
\put(9.5,1.1){\circle{0.6}}
\put(8.5,0.4){\circle*{0.6}}
\put(9.5,0.4){\circle*{0.6}}

\end{picture}
\end{subfigure}
%%%
\\
\begin{subfigure}[b]{0.01\textwidth}
\begin{picture}(0.5,8)
\put(0.5,4){$\supsetneq$}
\end{picture}
\end{subfigure}
%6
\begin{subfigure}[b]{0.3\textwidth}
\begin{picture}(10,8)
\put(0,7.5){\line(1,0){10}}
\put(0,6){\line(1,0){10}}
\put(2,4.5){\line(1,0){8}}
\put(4,3){\line(1,0){6}}
\put(6,1.5){\line(1,0){4}}
\put(8,0){\line(1,0){2}}
\put(10,7.5){\line(0,-1){7.5}}
\put(8,7.5){\line(0,-1){7.5}}
\put(6,7.5){\line(0,-1){6}}
\put(4,7.5){\line(0,-1){4.5}}
\put(2,7.5){\line(0,-1){3}}
\put(0,7.5){\line(0,-1){1.5}}

%1n2
%\put(0.5,7.1){\circle{0.6}}
%\put(0.5,7.1){\circle*{0.38}}
%\put(1.5,7.1){\circle{0.6}}
%\put(1.5,7.1){\circle*{0.38}}
\put(0.5,6.4){\circle*{0.6}}
\put(1.5,6.4){\circle*{0.6}}
%1n3
\put(2.5,7.1){\circle{0.6}}
\put(3.5,7.1){\circle*{0.6}}
\put(2.5,6.4){\circle*{0.6}}
\put(3.5,6.4){\circle{0.6}}
%1n4
\put(4.5,7.1){\circle{0.6}}
\put(5.5,7.1){\circle*{0.6}}
\put(4.5,6.4){\circle*{0.6}}
\put(5.5,6.4){\circle{0.6}}
%1n5
\put(6.5,7.1){\circle*{0.6}}
\put(7.5,7.1){\circle*{0.6}}
\put(6.5,6.4){\circle{0.6}}
\put(7.5,6.4){\circle{0.6}}
%1n6
%\put(8.5,7.1){\circle{0.6}}
%\put(8.5,7.1){\circle*{0.38}}
%\put(9.5,7.1){\circle{0.6}}
%\put(9.5,7.1){\circle*{0.38}}
\put(8.5,6.4){\circle*{0.6}}
\put(9.5,6.4){\circle*{0.6}}

%2n3
\put(2.5,5.6){\circle*{0.6}}
\put(3.5,5.6){\circle*{0.6}}
\put(2.5,4.9){\circle*{0.6}}
\put(3.5,4.9){\circle*{0.6}}
%2n4
\put(4.5,5.6){\circle*{0.6}}
\put(5.5,5.6){\circle*{0.6}}
\put(4.5,4.9){\circle*{0.6}}
\put(5.5,4.9){\circle*{0.6}}
%2n5
%\put(6.5,5.6){\circle{0.6}}
%\put(6.5,5.6){\circle*{0.38}}
\put(7.5,5.6){\circle{0.6}}
\put(6.5,4.9){\circle*{0.6}}
\put(7.5,4.9){\circle*{0.6}}
%2n6
\put(8.5,5.6){\circle*{0.6}}
\put(9.5,5.6){\circle*{0.6}}
\put(8.5,4.9){\circle{0.6}}
\put(9.5,4.9){\circle{0.6}}

%3n4
\put(4.5,4.1){\circle*{0.6}}
%\put(5.5,4.1){\circle{0.6}}
%\put(5.5,4.1){\circle*{0.38}}
\put(4.5,3.4){\circle*{0.6}}
\put(5.5,3.4){\circle*{0.6}}
%3n5
\put(6.5,4.1){\circle*{0.6}}
\put(7.5,4.1){\circle*{0.6}}
\put(6.5,3.4){\circle*{0.6}}
\put(7.5,3.4){\circle*{0.6}}
%3n6
\put(8.5,4.1){\circle*{0.6}}
\put(9.5,4.1){\circle*{0.6}}
\put(8.5,3.4){\circle*{0.6}}
\put(9.5,3.4){\circle*{0.6}}

%4n5
\put(6.5,2.6){\circle*{0.6}}
\put(7.5,2.6){\circle*{0.6}}
\put(6.5,1.9){\circle*{0.6}}
\put(7.5,1.9){\circle*{0.6}}
%4n6
\put(8.5,2.6){\circle*{0.6}}
\put(9.5,2.6){\circle*{0.6}}
\put(8.5,1.9){\circle*{0.6}}
\put(9.5,1.9){\circle*{0.6}}

%5n6
\put(8.5,1.1){\circle{0.6}}
\put(9.5,1.1){\circle{0.6}}
\put(8.5,0.4){\circle*{0.6}}
\put(9.5,0.4){\circle*{0.6}}

\end{picture}
\end{subfigure}
%%%
\begin{subfigure}[b]{0.01\textwidth}
\begin{picture}(0.5,8)
\put(0.5,4){$\supsetneq$}
\end{picture}
\end{subfigure}
%7
\begin{subfigure}[b]{0.3\textwidth}
\begin{picture}(10,8)
\put(0,7.5){\line(1,0){10}}
\put(0,6){\line(1,0){10}}
\put(2,4.5){\line(1,0){8}}
\put(4,3){\line(1,0){6}}
\put(6,1.5){\line(1,0){4}}
\put(8,0){\line(1,0){2}}
\put(10,7.5){\line(0,-1){7.5}}
\put(8,7.5){\line(0,-1){7.5}}
\put(6,7.5){\line(0,-1){6}}
\put(4,7.5){\line(0,-1){4.5}}
\put(2,7.5){\line(0,-1){3}}
\put(0,7.5){\line(0,-1){1.5}}

%1n2
%\put(0.5,7.1){\circle{0.6}}
%\put(0.5,7.1){\circle*{0.38}}
%\put(1.5,7.1){\circle{0.6}}
%\put(1.5,7.1){\circle*{0.38}}
\put(0.5,6.4){\circle*{0.6}}
\put(1.5,6.4){\circle*{0.6}}
%1n3
\put(2.5,7.1){\circle*{0.6}}
\put(3.5,7.1){\circle*{0.6}}
\put(2.5,6.4){\circle*{0.6}}
\put(3.5,6.4){\circle*{0.6}}
%1n4
\put(4.5,7.1){\circle*{0.6}}
\put(5.5,7.1){\circle*{0.6}}
\put(4.5,6.4){\circle*{0.6}}
\put(5.5,6.4){\circle*{0.6}}
%1n5
\put(6.5,7.1){\circle*{0.6}}
\put(7.5,7.1){\circle*{0.6}}
\put(6.5,6.4){\circle{0.6}}
\put(7.5,6.4){\circle{0.6}}
%1n6
%\put(8.5,7.1){\circle{0.6}}
%\put(8.5,7.1){\circle*{0.38}}
%\put(9.5,7.1){\circle{0.6}}
%\put(9.5,7.1){\circle*{0.38}}
\put(8.5,6.4){\circle*{0.6}}
\put(9.5,6.4){\circle*{0.6}}

%2n3
\put(2.5,5.6){\circle*{0.6}}
\put(3.5,5.6){\circle*{0.6}}
\put(2.5,4.9){\circle*{0.6}}
\put(3.5,4.9){\circle*{0.6}}
%2n4
\put(4.5,5.6){\circle*{0.6}}
\put(5.5,5.6){\circle*{0.6}}
\put(4.5,4.9){\circle*{0.6}}
\put(5.5,4.9){\circle*{0.6}}
%2n5
%\put(6.5,5.6){\circle{0.6}}
%\put(6.5,5.6){\circle*{0.38}}
%\put(7.5,5.6){\circle{0.6}}
%\put(7.5,5.6){\circle*{0.38}}
\put(6.5,4.9){\circle*{0.6}}
\put(7.5,4.9){\circle*{0.6}}
%2n6
\put(8.5,5.6){\circle*{0.6}}
\put(9.5,5.6){\circle*{0.6}}
\put(8.5,4.9){\circle{0.6}}
\put(9.5,4.9){\circle{0.6}}

%3n4
\put(4.5,4.1){\circle*{0.6}}
%\put(5.5,4.1){\circle{0.6}}
%\put(5.5,4.1){\circle*{0.38}}
\put(4.5,3.4){\circle*{0.6}}
\put(5.5,3.4){\circle*{0.6}}
%3n5
\put(6.5,4.1){\circle*{0.6}}
\put(7.5,4.1){\circle*{0.6}}
\put(6.5,3.4){\circle*{0.6}}
\put(7.5,3.4){\circle*{0.6}}
%3n6
\put(8.5,4.1){\circle*{0.6}}
\put(9.5,4.1){\circle*{0.6}}
\put(8.5,3.4){\circle*{0.6}}
\put(9.5,3.4){\circle*{0.6}}

%4n5
\put(6.5,2.6){\circle*{0.6}}
\put(7.5,2.6){\circle*{0.6}}
\put(6.5,1.9){\circle*{0.6}}
\put(7.5,1.9){\circle*{0.6}}
%4n6
\put(8.5,2.6){\circle*{0.6}}
\put(9.5,2.6){\circle*{0.6}}
\put(8.5,1.9){\circle*{0.6}}
\put(9.5,1.9){\circle*{0.6}}

%5n6
\put(8.5,1.1){\circle{0.6}}
\put(9.5,1.1){\circle{0.6}}
\put(8.5,0.4){\circle*{0.6}}
\put(9.5,0.4){\circle*{0.6}}

\end{picture}
\end{subfigure}
%%%
\begin{subfigure}[b]{0.01\textwidth}
\begin{picture}(0.5,8)
\put(0.5,4){$\supsetneq$}
\end{picture}
\end{subfigure}
%8
\begin{subfigure}[b]{0.3\textwidth}
\begin{picture}(10,8)
\put(0,7.5){\line(1,0){10}}
\put(0,6){\line(1,0){10}}
\put(2,4.5){\line(1,0){8}}
\put(4,3){\line(1,0){6}}
\put(6,1.5){\line(1,0){4}}
\put(8,0){\line(1,0){2}}
\put(10,7.5){\line(0,-1){7.5}}
\put(8,7.5){\line(0,-1){7.5}}
\put(6,7.5){\line(0,-1){6}}
\put(4,7.5){\line(0,-1){4.5}}
\put(2,7.5){\line(0,-1){3}}
\put(0,7.5){\line(0,-1){1.5}}

%1n2
%\put(0.5,7.1){\circle{0.6}}
%\put(0.5,7.1){\circle*{0.38}}
%\put(1.5,7.1){\circle{0.6}}
%\put(1.5,7.1){\circle*{0.38}}
\put(0.5,6.4){\circle*{0.6}}
\put(1.5,6.4){\circle*{0.6}}
%1n3
\put(2.5,7.1){\circle*{0.6}}
\put(3.5,7.1){\circle*{0.6}}
\put(2.5,6.4){\circle*{0.6}}
\put(3.5,6.4){\circle*{0.6}}
%1n4
\put(4.5,7.1){\circle*{0.6}}
\put(5.5,7.1){\circle*{0.6}}
\put(4.5,6.4){\circle*{0.6}}
\put(5.5,6.4){\circle*{0.6}}
%1n5
\put(6.5,7.1){\circle*{0.6}}
\put(7.5,7.1){\circle*{0.6}}
%\put(6.5,6.4){\circle{0.6}}
%\put(6.5,6.4){\circle*{0.38}}
\put(7.5,6.4){\circle{0.6}}
%1n6
%\put(8.5,7.1){\circle{0.6}}
%\put(8.5,7.1){\circle*{0.38}}
%\put(9.5,7.1){\circle{0.6}}
%\put(9.5,7.1){\circle*{0.38}}
\put(8.5,6.4){\circle*{0.6}}
\put(9.5,6.4){\circle*{0.6}}

%2n3
\put(2.5,5.6){\circle*{0.6}}
\put(3.5,5.6){\circle*{0.6}}
\put(2.5,4.9){\circle*{0.6}}
\put(3.5,4.9){\circle*{0.6}}
%2n4
\put(4.5,5.6){\circle*{0.6}}
\put(5.5,5.6){\circle*{0.6}}
\put(4.5,4.9){\circle*{0.6}}
\put(5.5,4.9){\circle*{0.6}}
%2n5
%\put(6.5,5.6){\circle{0.6}}
%\put(6.5,5.6){\circle*{0.38}}
%\put(7.5,5.6){\circle{0.6}}
%\put(7.5,5.6){\circle*{0.38}}
\put(6.5,4.9){\circle*{0.6}}
\put(7.5,4.9){\circle*{0.6}}
%2n6
\put(8.5,5.6){\circle*{0.6}}
\put(9.5,5.6){\circle*{0.6}}
\put(8.5,4.9){\circle{0.6}}
\put(9.5,4.9){\circle{0.6}}

%3n4
\put(4.5,4.1){\circle*{0.6}}
%\put(5.5,4.1){\circle{0.6}}
%\put(5.5,4.1){\circle*{0.38}}
\put(4.5,3.4){\circle*{0.6}}
\put(5.5,3.4){\circle*{0.6}}
%3n5
\put(6.5,4.1){\circle*{0.6}}
\put(7.5,4.1){\circle*{0.6}}
\put(6.5,3.4){\circle*{0.6}}
\put(7.5,3.4){\circle*{0.6}}
%3n6
\put(8.5,4.1){\circle*{0.6}}
\put(9.5,4.1){\circle*{0.6}}
\put(8.5,3.4){\circle*{0.6}}
\put(9.5,3.4){\circle*{0.6}}

%4n5
\put(6.5,2.6){\circle*{0.6}}
\put(7.5,2.6){\circle*{0.6}}
\put(6.5,1.9){\circle*{0.6}}
\put(7.5,1.9){\circle*{0.6}}
%4n6
\put(8.5,2.6){\circle*{0.6}}
\put(9.5,2.6){\circle*{0.6}}
\put(8.5,1.9){\circle*{0.6}}
\put(9.5,1.9){\circle*{0.6}}

%5n6
\put(8.5,1.1){\circle{0.6}}
\put(9.5,1.1){\circle{0.6}}
\put(8.5,0.4){\circle*{0.6}}
\put(9.5,0.4){\circle*{0.6}}

\end{picture}
\end{subfigure}
%%%
\\
\begin{subfigure}[b]{0.01\textwidth}
\begin{picture}(0.5,8)
\put(0.5,4){$\supsetneq$}
\end{picture}
\end{subfigure}
%9
\begin{subfigure}[b]{0.3\textwidth}
\begin{picture}(10,8)
\put(0,7.5){\line(1,0){10}}
\put(0,6){\line(1,0){10}}
\put(2,4.5){\line(1,0){8}}
\put(4,3){\line(1,0){6}}
\put(6,1.5){\line(1,0){4}}
\put(8,0){\line(1,0){2}}
\put(10,7.5){\line(0,-1){7.5}}
\put(8,7.5){\line(0,-1){7.5}}
\put(6,7.5){\line(0,-1){6}}
\put(4,7.5){\line(0,-1){4.5}}
\put(2,7.5){\line(0,-1){3}}
\put(0,7.5){\line(0,-1){1.5}}

%1n2
%\put(0.5,7.1){\circle{0.6}}
%\put(0.5,7.1){\circle*{0.38}}
%\put(1.5,7.1){\circle{0.6}}
%\put(1.5,7.1){\circle*{0.38}}
\put(0.5,6.4){\circle*{0.6}}
\put(1.5,6.4){\circle*{0.6}}
%1n3
\put(2.5,7.1){\circle*{0.6}}
\put(3.5,7.1){\circle*{0.6}}
\put(2.5,6.4){\circle*{0.6}}
\put(3.5,6.4){\circle*{0.6}}
%1n4
\put(4.5,7.1){\circle*{0.6}}
\put(5.5,7.1){\circle*{0.6}}
\put(4.5,6.4){\circle*{0.6}}
\put(5.5,6.4){\circle*{0.6}}
%1n5
\put(6.5,7.1){\circle*{0.6}}
\put(7.5,7.1){\circle*{0.6}}
%\put(6.5,6.4){\circle{0.6}}
%\put(6.5,6.4){\circle*{0.38}}
%\put(7.5,6.4){\circle{0.6}}
%\put(7.5,6.4){\circle*{0.38}}
%1n6
%\put(8.5,7.1){\circle{0.6}}
%\put(8.5,7.1){\circle*{0.38}}
%\put(9.5,7.1){\circle{0.6}}
%\put(9.5,7.1){\circle*{0.38}}
\put(8.5,6.4){\circle*{0.6}}
\put(9.5,6.4){\circle*{0.6}}

%2n3
\put(2.5,5.6){\circle*{0.6}}
\put(3.5,5.6){\circle*{0.6}}
\put(2.5,4.9){\circle*{0.6}}
\put(3.5,4.9){\circle*{0.6}}
%2n4
\put(4.5,5.6){\circle*{0.6}}
\put(5.5,5.6){\circle*{0.6}}
\put(4.5,4.9){\circle*{0.6}}
\put(5.5,4.9){\circle*{0.6}}
%2n5
%\put(6.5,5.6){\circle{0.6}}
%\put(6.5,5.6){\circle*{0.38}}
%\put(7.5,5.6){\circle{0.6}}
%\put(7.5,5.6){\circle*{0.38}}
\put(6.5,4.9){\circle*{0.6}}
\put(7.5,4.9){\circle*{0.6}}
%2n6
\put(8.5,5.6){\circle*{0.6}}
\put(9.5,5.6){\circle*{0.6}}
\put(8.5,4.9){\circle{0.6}}
\put(9.5,4.9){\circle{0.6}}

%3n4
\put(4.5,4.1){\circle*{0.6}}
%\put(5.5,4.1){\circle{0.6}}
%\put(5.5,4.1){\circle*{0.38}}
\put(4.5,3.4){\circle*{0.6}}
\put(5.5,3.4){\circle*{0.6}}
%3n5
\put(6.5,4.1){\circle*{0.6}}
\put(7.5,4.1){\circle*{0.6}}
\put(6.5,3.4){\circle*{0.6}}
\put(7.5,3.4){\circle*{0.6}}
%3n6
\put(8.5,4.1){\circle*{0.6}}
\put(9.5,4.1){\circle*{0.6}}
\put(8.5,3.4){\circle*{0.6}}
\put(9.5,3.4){\circle*{0.6}}

%4n5
\put(6.5,2.6){\circle*{0.6}}
\put(7.5,2.6){\circle*{0.6}}
\put(6.5,1.9){\circle*{0.6}}
\put(7.5,1.9){\circle*{0.6}}
%4n6
\put(8.5,2.6){\circle*{0.6}}
\put(9.5,2.6){\circle*{0.6}}
\put(8.5,1.9){\circle*{0.6}}
\put(9.5,1.9){\circle*{0.6}}

%5n6
\put(8.5,1.1){\circle{0.6}}
\put(9.5,1.1){\circle{0.6}}
\put(8.5,0.4){\circle*{0.6}}
\put(9.5,0.4){\circle*{0.6}}

\end{picture}
\end{subfigure}
%%%
\begin{subfigure}[b]{0.01\textwidth}
\begin{picture}(0.5,8)
\put(0.5,4){$\supsetneq$}
\end{picture}
\end{subfigure}
%10
\begin{subfigure}[b]{0.3\textwidth}
\begin{picture}(10,8)
\put(0,7.5){\line(1,0){10}}
\put(0,6){\line(1,0){10}}
\put(2,4.5){\line(1,0){8}}
\put(4,3){\line(1,0){6}}
\put(6,1.5){\line(1,0){4}}
\put(8,0){\line(1,0){2}}
\put(10,7.5){\line(0,-1){7.5}}
\put(8,7.5){\line(0,-1){7.5}}
\put(6,7.5){\line(0,-1){6}}
\put(4,7.5){\line(0,-1){4.5}}
\put(2,7.5){\line(0,-1){3}}
\put(0,7.5){\line(0,-1){1.5}}

%1n2
%\put(0.5,7.1){\circle{0.6}}
%\put(0.5,7.1){\circle*{0.38}}
%\put(1.5,7.1){\circle{0.6}}
%\put(1.5,7.1){\circle*{0.38}}
\put(0.5,6.4){\circle*{0.6}}
\put(1.5,6.4){\circle*{0.6}}
%1n3
\put(2.5,7.1){\circle*{0.6}}
\put(3.5,7.1){\circle*{0.6}}
\put(2.5,6.4){\circle*{0.6}}
\put(3.5,6.4){\circle*{0.6}}
%1n4
\put(4.5,7.1){\circle*{0.6}}
\put(5.5,7.1){\circle*{0.6}}
\put(4.5,6.4){\circle*{0.6}}
\put(5.5,6.4){\circle*{0.6}}
%1n5
\put(6.5,7.1){\circle*{0.6}}
\put(7.5,7.1){\circle*{0.6}}
%\put(6.5,6.4){\circle{0.6}}
%\put(6.5,6.4){\circle*{0.38}}
%\put(7.5,6.4){\circle{0.6}}
%\put(7.5,6.4){\circle*{0.38}}
%1n6
%\put(8.5,7.1){\circle{0.6}}
%\put(8.5,7.1){\circle*{0.38}}
%\put(9.5,7.1){\circle{0.6}}
%\put(9.5,7.1){\circle*{0.38}}
\put(8.5,6.4){\circle*{0.6}}
\put(9.5,6.4){\circle*{0.6}}

%2n3
\put(2.5,5.6){\circle*{0.6}}
\put(3.5,5.6){\circle*{0.6}}
\put(2.5,4.9){\circle*{0.6}}
\put(3.5,4.9){\circle*{0.6}}
%2n4
\put(4.5,5.6){\circle*{0.6}}
\put(5.5,5.6){\circle*{0.6}}
\put(4.5,4.9){\circle*{0.6}}
\put(5.5,4.9){\circle*{0.6}}
%2n5
%\put(6.5,5.6){\circle{0.6}}
%\put(6.5,5.6){\circle*{0.38}}
%\put(7.5,5.6){\circle{0.6}}
%\put(7.5,5.6){\circle*{0.38}}
\put(6.5,4.9){\circle*{0.6}}
\put(7.5,4.9){\circle*{0.6}}
%2n6
\put(8.5,5.6){\circle*{0.6}}
\put(9.5,5.6){\circle*{0.6}}
%\put(8.5,4.9){\circle{0.6}}
%\put(8.5,4.9){\circle*{0.38}}
\put(9.5,4.9){\circle{0.6}}

%3n4
\put(4.5,4.1){\circle*{0.6}}
%\put(5.5,4.1){\circle{0.6}}
%\put(5.5,4.1){\circle*{0.38}}
\put(4.5,3.4){\circle*{0.6}}
\put(5.5,3.4){\circle*{0.6}}
%3n5
\put(6.5,4.1){\circle*{0.6}}
\put(7.5,4.1){\circle*{0.6}}
\put(6.5,3.4){\circle*{0.6}}
\put(7.5,3.4){\circle*{0.6}}
%3n6
\put(8.5,4.1){\circle*{0.6}}
\put(9.5,4.1){\circle*{0.6}}
\put(8.5,3.4){\circle*{0.6}}
\put(9.5,3.4){\circle*{0.6}}

%4n5
\put(6.5,2.6){\circle*{0.6}}
\put(7.5,2.6){\circle*{0.6}}
\put(6.5,1.9){\circle*{0.6}}
\put(7.5,1.9){\circle*{0.6}}
%4n6
\put(8.5,2.6){\circle*{0.6}}
\put(9.5,2.6){\circle*{0.6}}
\put(8.5,1.9){\circle*{0.6}}
\put(9.5,1.9){\circle*{0.6}}

%5n6
\put(8.5,1.1){\circle{0.6}}
\put(9.5,1.1){\circle{0.6}}
\put(8.5,0.4){\circle*{0.6}}
\put(9.5,0.4){\circle*{0.6}}

\end{picture}
\end{subfigure}
%%%
\begin{subfigure}[b]{0.01\textwidth}
\begin{picture}(0.5,8)
\put(0.5,4){$\supsetneq$}
\end{picture}
\end{subfigure}
%11
\begin{subfigure}[b]{0.3\textwidth}
\begin{picture}(10,8)
\put(0,7.5){\line(1,0){10}}
\put(0,6){\line(1,0){10}}
\put(2,4.5){\line(1,0){8}}
\put(4,3){\line(1,0){6}}
\put(6,1.5){\line(1,0){4}}
\put(8,0){\line(1,0){2}}
\put(10,7.5){\line(0,-1){7.5}}
\put(8,7.5){\line(0,-1){7.5}}
\put(6,7.5){\line(0,-1){6}}
\put(4,7.5){\line(0,-1){4.5}}
\put(2,7.5){\line(0,-1){3}}
\put(0,7.5){\line(0,-1){1.5}}

%1n2
%\put(0.5,7.1){\circle{0.6}}
%\put(0.5,7.1){\circle*{0.38}}
%\put(1.5,7.1){\circle{0.6}}
%\put(1.5,7.1){\circle*{0.38}}
\put(0.5,6.4){\circle*{0.6}}
\put(1.5,6.4){\circle*{0.6}}
%1n3
\put(2.5,7.1){\circle*{0.6}}
\put(3.5,7.1){\circle*{0.6}}
\put(2.5,6.4){\circle*{0.6}}
\put(3.5,6.4){\circle*{0.6}}
%1n4
\put(4.5,7.1){\circle*{0.6}}
\put(5.5,7.1){\circle*{0.6}}
\put(4.5,6.4){\circle*{0.6}}
\put(5.5,6.4){\circle*{0.6}}
%1n5
\put(6.5,7.1){\circle*{0.6}}
\put(7.5,7.1){\circle*{0.6}}
%\put(6.5,6.4){\circle{0.6}}
%\put(6.5,6.4){\circle*{0.38}}
%\put(7.5,6.4){\circle{0.6}}
%\put(7.5,6.4){\circle*{0.38}}
%1n6
%\put(8.5,7.1){\circle{0.6}}
%\put(8.5,7.1){\circle*{0.38}}
%\put(9.5,7.1){\circle{0.6}}
%\put(9.5,7.1){\circle*{0.38}}
\put(8.5,6.4){\circle*{0.6}}
\put(9.5,6.4){\circle*{0.6}}

%2n3
\put(2.5,5.6){\circle*{0.6}}
\put(3.5,5.6){\circle*{0.6}}
\put(2.5,4.9){\circle*{0.6}}
\put(3.5,4.9){\circle*{0.6}}
%2n4
\put(4.5,5.6){\circle*{0.6}}
\put(5.5,5.6){\circle*{0.6}}
\put(4.5,4.9){\circle*{0.6}}
\put(5.5,4.9){\circle*{0.6}}
%2n5
%\put(6.5,5.6){\circle{0.6}}
%\put(6.5,5.6){\circle*{0.38}}
%\put(7.5,5.6){\circle{0.6}}
%\put(7.5,5.6){\circle*{0.38}}
\put(6.5,4.9){\circle*{0.6}}
\put(7.5,4.9){\circle*{0.6}}
%2n6
\put(8.5,5.6){\circle*{0.6}}
\put(9.5,5.6){\circle*{0.6}}
%\put(8.5,4.9){\circle{0.6}}
%\put(8.5,4.9){\circle*{0.38}}
%\put(9.5,4.9){\circle{0.6}}
%\put(9.5,4.9){\circle*{0.38}}

%3n4
\put(4.5,4.1){\circle*{0.6}}
%\put(5.5,4.1){\circle{0.6}}
%\put(5.5,4.1){\circle*{0.38}}
\put(4.5,3.4){\circle*{0.6}}
\put(5.5,3.4){\circle*{0.6}}
%3n5
\put(6.5,4.1){\circle*{0.6}}
\put(7.5,4.1){\circle*{0.6}}
\put(6.5,3.4){\circle*{0.6}}
\put(7.5,3.4){\circle*{0.6}}
%3n6
\put(8.5,4.1){\circle*{0.6}}
\put(9.5,4.1){\circle*{0.6}}
\put(8.5,3.4){\circle*{0.6}}
\put(9.5,3.4){\circle*{0.6}}

%4n5
\put(6.5,2.6){\circle*{0.6}}
\put(7.5,2.6){\circle*{0.6}}
\put(6.5,1.9){\circle*{0.6}}
\put(7.5,1.9){\circle*{0.6}}
%4n6
\put(8.5,2.6){\circle*{0.6}}
\put(9.5,2.6){\circle*{0.6}}
\put(8.5,1.9){\circle*{0.6}}
\put(9.5,1.9){\circle*{0.6}}

%5n6
\put(8.5,1.1){\circle{0.6}}
\put(9.5,1.1){\circle{0.6}}
\put(8.5,0.4){\circle*{0.6}}
\put(9.5,0.4){\circle*{0.6}}

\end{picture}
\end{subfigure}
%%%
\\
\begin{subfigure}[b]{0.01\textwidth}
\begin{picture}(0.5,8)
\put(0.5,4){$\supsetneq$}
\end{picture}
\end{subfigure}
%12
\begin{subfigure}[b]{0.3\textwidth}
\begin{picture}(10,8)
\put(0,7.5){\line(1,0){10}}
\put(0,6){\line(1,0){10}}
\put(2,4.5){\line(1,0){8}}
\put(4,3){\line(1,0){6}}
\put(6,1.5){\line(1,0){4}}
\put(8,0){\line(1,0){2}}
\put(10,7.5){\line(0,-1){7.5}}
\put(8,7.5){\line(0,-1){7.5}}
\put(6,7.5){\line(0,-1){6}}
\put(4,7.5){\line(0,-1){4.5}}
\put(2,7.5){\line(0,-1){3}}
\put(0,7.5){\line(0,-1){1.5}}

%1n2
%\put(0.5,7.1){\circle{0.6}}
%\put(0.5,7.1){\circle*{0.38}}
%\put(1.5,7.1){\circle{0.6}}
%\put(1.5,7.1){\circle*{0.38}}
\put(0.5,6.4){\circle*{0.6}}
\put(1.5,6.4){\circle*{0.6}}
%1n3
\put(2.5,7.1){\circle*{0.6}}
\put(3.5,7.1){\circle*{0.6}}
\put(2.5,6.4){\circle*{0.6}}
\put(3.5,6.4){\circle*{0.6}}
%1n4
\put(4.5,7.1){\circle*{0.6}}
\put(5.5,7.1){\circle*{0.6}}
\put(4.5,6.4){\circle*{0.6}}
\put(5.5,6.4){\circle*{0.6}}
%1n5
\put(6.5,7.1){\circle*{0.6}}
\put(7.5,7.1){\circle*{0.6}}
%\put(6.5,6.4){\circle{0.6}}
%\put(6.5,6.4){\circle*{0.38}}
%\put(7.5,6.4){\circle{0.6}}
%\put(7.5,6.4){\circle*{0.38}}
%1n6
%\put(8.5,7.1){\circle{0.6}}
%\put(8.5,7.1){\circle*{0.38}}
%\put(9.5,7.1){\circle{0.6}}
%\put(9.5,7.1){\circle*{0.38}}
\put(8.5,6.4){\circle*{0.6}}
\put(9.5,6.4){\circle*{0.6}}

%2n3
\put(2.5,5.6){\circle*{0.6}}
\put(3.5,5.6){\circle*{0.6}}
\put(2.5,4.9){\circle*{0.6}}
\put(3.5,4.9){\circle*{0.6}}
%2n4
\put(4.5,5.6){\circle*{0.6}}
\put(5.5,5.6){\circle*{0.6}}
\put(4.5,4.9){\circle*{0.6}}
\put(5.5,4.9){\circle*{0.6}}
%2n5
%\put(6.5,5.6){\circle{0.6}}
%\put(6.5,5.6){\circle*{0.38}}
%\put(7.5,5.6){\circle{0.6}}
%\put(7.5,5.6){\circle*{0.38}}
\put(6.5,4.9){\circle*{0.6}}
\put(7.5,4.9){\circle*{0.6}}
%2n6
\put(8.5,5.6){\circle*{0.6}}
\put(9.5,5.6){\circle*{0.6}}
%\put(8.5,4.9){\circle{0.6}}
%\put(8.5,4.9){\circle*{0.38}}
%\put(9.5,4.9){\circle{0.6}}
%\put(9.5,4.9){\circle*{0.38}}

%3n4
\put(4.5,4.1){\circle*{0.6}}
%\put(5.5,4.1){\circle{0.6}}
%\put(5.5,4.1){\circle*{0.38}}
\put(4.5,3.4){\circle*{0.6}}
\put(5.5,3.4){\circle*{0.6}}
%3n5
\put(6.5,4.1){\circle*{0.6}}
\put(7.5,4.1){\circle*{0.6}}
\put(6.5,3.4){\circle*{0.6}}
\put(7.5,3.4){\circle*{0.6}}
%3n6
\put(8.5,4.1){\circle*{0.6}}
\put(9.5,4.1){\circle*{0.6}}
\put(8.5,3.4){\circle*{0.6}}
\put(9.5,3.4){\circle*{0.6}}

%4n5
\put(6.5,2.6){\circle*{0.6}}
\put(7.5,2.6){\circle*{0.6}}
\put(6.5,1.9){\circle*{0.6}}
\put(7.5,1.9){\circle*{0.6}}
%4n6
\put(8.5,2.6){\circle*{0.6}}
\put(9.5,2.6){\circle*{0.6}}
\put(8.5,1.9){\circle*{0.6}}
\put(9.5,1.9){\circle*{0.6}}

%5n6
%\put(8.5,1.1){\circle{0.6}}
%\put(8.5,1.1){\circle*{0.38}}
\put(9.5,1.1){\circle{0.6}}
\put(8.5,0.4){\circle*{0.6}}
\put(9.5,0.4){\circle*{0.6}}

\end{picture}
\end{subfigure}
%%%
\begin{subfigure}[b]{0.01\textwidth}
\begin{picture}(0.5,8)
\put(0.5,4){$\supsetneq$}
\end{picture}
\end{subfigure}
%13
\begin{subfigure}[b]{0.3\textwidth}
\begin{picture}(10,8)
\put(0,7.5){\line(1,0){10}}
\put(0,6){\line(1,0){10}}
\put(2,4.5){\line(1,0){8}}
\put(4,3){\line(1,0){6}}
\put(6,1.5){\line(1,0){4}}
\put(8,0){\line(1,0){2}}
\put(10,7.5){\line(0,-1){7.5}}
\put(8,7.5){\line(0,-1){7.5}}
\put(6,7.5){\line(0,-1){6}}
\put(4,7.5){\line(0,-1){4.5}}
\put(2,7.5){\line(0,-1){3}}
\put(0,7.5){\line(0,-1){1.5}}

%1n2
%\put(0.5,7.1){\circle{0.6}}
%\put(0.5,7.1){\circle*{0.38}}
%\put(1.5,7.1){\circle{0.6}}
%\put(1.5,7.1){\circle*{0.38}}
\put(0.5,6.4){\circle*{0.6}}
\put(1.5,6.4){\circle*{0.6}}
%1n3
\put(2.5,7.1){\circle*{0.6}}
\put(3.5,7.1){\circle*{0.6}}
\put(2.5,6.4){\circle*{0.6}}
\put(3.5,6.4){\circle*{0.6}}
%1n4
\put(4.5,7.1){\circle*{0.6}}
\put(5.5,7.1){\circle*{0.6}}
\put(4.5,6.4){\circle*{0.6}}
\put(5.5,6.4){\circle*{0.6}}
%1n5
\put(6.5,7.1){\circle*{0.6}}
\put(7.5,7.1){\circle*{0.6}}
%\put(6.5,6.4){\circle{0.6}}
%\put(6.5,6.4){\circle*{0.38}}
%\put(7.5,6.4){\circle{0.6}}
%\put(7.5,6.4){\circle*{0.38}}
%1n6
%\put(8.5,7.1){\circle{0.6}}
%\put(8.5,7.1){\circle*{0.38}}
%\put(9.5,7.1){\circle{0.6}}
%\put(9.5,7.1){\circle*{0.38}}
\put(8.5,6.4){\circle*{0.6}}
\put(9.5,6.4){\circle*{0.6}}

%2n3
\put(2.5,5.6){\circle*{0.6}}
\put(3.5,5.6){\circle*{0.6}}
\put(2.5,4.9){\circle*{0.6}}
\put(3.5,4.9){\circle*{0.6}}
%2n4
\put(4.5,5.6){\circle*{0.6}}
\put(5.5,5.6){\circle*{0.6}}
\put(4.5,4.9){\circle*{0.6}}
\put(5.5,4.9){\circle*{0.6}}
%2n5
%\put(6.5,5.6){\circle{0.6}}
%\put(6.5,5.6){\circle*{0.38}}
%\put(7.5,5.6){\circle{0.6}}
%\put(7.5,5.6){\circle*{0.38}}
\put(6.5,4.9){\circle*{0.6}}
\put(7.5,4.9){\circle*{0.6}}
%2n6
\put(8.5,5.6){\circle*{0.6}}
\put(9.5,5.6){\circle*{0.6}}
%\put(8.5,4.9){\circle{0.6}}
%\put(8.5,4.9){\circle*{0.38}}
%\put(9.5,4.9){\circle{0.6}}
%\put(9.5,4.9){\circle*{0.38}}

%3n4
\put(4.5,4.1){\circle*{0.6}}
%\put(5.5,4.1){\circle{0.6}}
%\put(5.5,4.1){\circle*{0.38}}
\put(4.5,3.4){\circle*{0.6}}
\put(5.5,3.4){\circle*{0.6}}
%3n5
\put(6.5,4.1){\circle*{0.6}}
\put(7.5,4.1){\circle*{0.6}}
\put(6.5,3.4){\circle*{0.6}}
\put(7.5,3.4){\circle*{0.6}}
%3n6
\put(8.5,4.1){\circle*{0.6}}
\put(9.5,4.1){\circle*{0.6}}
\put(8.5,3.4){\circle*{0.6}}
\put(9.5,3.4){\circle*{0.6}}

%4n5
\put(6.5,2.6){\circle*{0.6}}
\put(7.5,2.6){\circle*{0.6}}
\put(6.5,1.9){\circle*{0.6}}
\put(7.5,1.9){\circle*{0.6}}
%4n6
\put(8.5,2.6){\circle*{0.6}}
\put(9.5,2.6){\circle*{0.6}}
\put(8.5,1.9){\circle*{0.6}}
\put(9.5,1.9){\circle*{0.6}}

%5n6
%\put(8.5,1.1){\circle{0.6}}
%\put(8.5,1.1){\circle*{0.38}}
%\put(9.5,1.1){\circle{0.6}}
%\put(9.5,1.1){\circle*{0.38}}
\put(8.5,0.4){\circle*{0.6}}
\put(9.5,0.4){\circle*{0.6}}

\end{picture}
\end{subfigure}
%%%
%%%
\caption{The change of hyperplanes which can be removed along a free filtration from $\An{A}$ to $\tilde{\An{A}}$.}%
\label{Fig_Rem_M31}
\end{figure}

\end{example}

Now we can prove Proposition \ref{prop:G31_NIF}:

\begin{proof}[Proof of Proposition \ref{prop:G31_NIF}]
Let $\tilde{\A}$ be a free filtration subarrangement.

If $\A(\crg{29}) \nsubseteq \tilde{\A}$, then with Lemma \ref{lem:FFSA_G31}, $\Betrag{\tilde{\A}} \geq 47$.

Now assume that $\tilde{\A} \cong \A(\crg{29})$.
In Lemma \ref{lem:G29_G31_arbt_desc} we saw, that $\tilde{\A}$ is a free filtration subarrangement.

In \cite[Remark~2.17]{2012arXiv1208.3131H} it is shown 
that one cannot remove a single hyperplane from $\A(\crg{29}) = \An{B}$ resulting in a free arrangement $\An{B}'$,
so there is no smaller free filtration subarrangement of $\A$. 
\end{proof}

\subsubsection{The reflection arrangements $\A(\crg{29})$ and $\A(\crg{31})$ are not recursivley free}
\label{subsubsec:A29_A31_nRF_e}

Let $\A := \A(W)$ be the reflection arrangement of the complex reflection group $W = \crg{31}$ and
$\An{B} := \A(W)$ the reflection arrangement of the complex reflection group $W = \crg{29}$.
As we saw in the previous section $\B \subsetneq \A$ is a \FFSAp.

We use the characterization of all free filtration subarrangements $\tilde{\A} \subseteq \A$ 
from Lemma \ref{lem:FFSA_G31}
and show that for all these subarrangements, there exists no hyperplane $H$ outside of $\A$ we can add to $\tilde{\A}$,
such that the resulting arrangement $\tilde{\A}\dot{\cup} \{H\}$ is free.

First we show, that it is not possible for $\tilde{\A} = \A$:

\begin{lemma}\label{lem:Add_H_A31}
There is no way to add a new hyperplane $H$ to $\A$ such that the arrangement
$\tilde{\A} := \A \dot{\cup} \{ H \}$ is free.
\end{lemma}

\begin{proof}
The exponents of $\A$ are $\expA{\A}{1,13,17,29}$.
Inspection of the intersection lattice $L:=\LA{\A}$ gives the following multisets of invariants:
\begin{equation}
\{\{ \Betrag{\A_X} \mid X \in L_2 \}\} = \{\{2^{360}, 3^{320}, 6^{30} \}\} \label{eq_lem:Add_H_A31}.
\end{equation}
Now assume that there exists a new hyperplane $H$ which we can add to $\A$ such that
$\tilde{\A} := \A \dot{\cup} \{ H \}$ is free.
Then by Lemma \ref{lem:A_u_H} we have $\sum_{X \in P_H} (\Betrag{\A_X}-1) \in \expAA{\A}$
where $P_H = \{ X \in L_2 \mid X \subseteq H \}$.
Hence with (\ref{eq_lem:Add_H_A31}) $H$ contains at least $4$ different rank $2$ subspaces
 (e.g.\ $13 = (6-1) + (6-1) + (3-1) +(2-1)$) from the intersection lattice.

But up to symmetry there are no more than $5$ possibilities to get a hyperplane $H$ with 
$\Betrag{ \{ X \in L_2 \mid X \subseteq H \} } \geq 3$ such that $\CharPolyA{\tilde{\A}}$ factors over the integers, 
but in each case $\CharPolyA{\tilde{\A}} = (t-1)(t-15)(t-16)(t-29)$, so with Theorem \ref{thm:A_AoH_exp} 
$\tilde{\A}$ can not be free. 
\end{proof}

Now we will prove that for all \FFSAs $\tilde{\A} \subseteq \A$ 
(see definition \ref{def:FFSA})
there exists no other hyperplane $H \notin \A$ we can add to $\tilde{\A}$ such that $\tilde{\A}\dot{\cup}\{H\}$ is
free.

\begin{lemma}\label{lem:Add_H_FFSA}
Let $\tilde{\A} \subseteq \A$ be a \FFSAp.
Let $H$ be a new hyperplane such that $\tilde{\A} \dot{\cup} \{ H \}$ is free.
Then $H \in \A$.
\end{lemma}

\begin{proof}
In Lemma \ref{lem:FFSA_G31} we have shown, that $\tilde{\A}$ is free with exponents 
$\expA{\tilde{\A}}{1,13,17,29-n}$, $n \leq 20$.
Let $L =  \LA{\A}$ and $\tilde{L} = \LA{\tilde{\A}} \subseteq L$.
We once more use the following multiset of invariants:
\begin{equation*}
\{\{ \Betrag{\A_X} \mid X \in L_2 \}\} = \{\{2^{360}, 3^{320}, 6^{30} \}\}.
\end{equation*}
Thus for $X \in \tilde{L}_2$ we have $2 \leq \Betrag{\tilde{\A}_X} \leq 6$. 

Suppose we add a new hyperplane $H$ such that $\tilde{\A} \dot{\cup} \{H\}$ is free.
Then by Lemma \ref{lem:A_u_H} we have $\sum_{X \in P_H} (\Betrag{\tilde{\A}_X} - 1) \in \expAA{\tilde{\A}}$
where $P_H = \{X \in \tilde{L}_2 \mid X \subseteq H \}$.

We immediately see that $\Betrag{P_H} \geq 3$ and if  $\Betrag{P_H} \in \{3,4\}$ there must be at least
two $X \in P_H$ with $\Betrag{\tilde{\A}_X} \geq 4$ or $\Betrag{\A_X} = 6$.
But for $X,Y \in L_2$, $X \neq Y$, with $\Betrag{\A_X} = \Betrag{\A_Y} = 6$ we either have $X+Y = V$  or
$X \subseteq K$ and $Y \subseteq K$ for a $K \in \A$. Hence in this case $H \in \A$.

Now assume that $\Betrag{P_H} \geq 5$ and there is at most one $X \in P_H$ with $\Betrag{\tilde{\A}_X} \geq 4$ 
or $\Betrag{\A_X} = 6$.
Then there are either at least three $X \in P_H$ with $\Betrag{\tilde{\A}_X} = 3$ or
at least four $X \in P_H$ with $\Betrag{\tilde{\A}_X} = 2$.
But in both cases with the same argument as above we must have $H \in \A$.

This finishes the proof.
\end{proof}

We close this section with the following corollary which completes the proof of Theorem \ref{thm:G31_G29_nRF}.
\begin{corollary}\label{coro:A31_FFSA_nRF}
Let $\tilde{\A} \subseteq \A$ be a \FFSA of $\A = \A(\crg{31})$.
Then $\tilde{\A}$ is not recursively free and in particular $\A(\crg{31})$ and $\A(\crg{29})$ are not recursively free.
\end{corollary}

\begin{proof}
The statement follows immediately from Lemma \ref{lem:Add_H_FFSA} and Proposition \ref{prop:G31_NIF}.
\end{proof}

\subsection{The reflection arrangement $\Arr{\crg{33}}$}

In this section we will see, that the reflection arrangement $\A(W)$ with $W$ isomorphic to the finite complex reflection group
$\crg{33}$ is not recursively free.

\begin{lemma}\label{lem:A33_nRF}
Let $\A = \A(W)$ be the reflection arrangement with $W \cong \crg{33}$.
Then $\A$ is not recursively free.
\end{lemma}
\begin{proof}
With Theorem \ref{thm:RA_IF} the reflection arrangement $\A$ is
not inductively free.

In \cite[Remark~2.17]{2012arXiv1208.3131H} it is shown 
that one cannot remove a single hyperplane from $\A$ resulting in a free arrangement $\A'$

Thus to prove the lemma, we have to show, that we also cannot add a new hyperplane $H$ such that the 
arrangements $\tilde{\A} := \A \dot{\cup} \{ H\}$ and $\tilde{\A}^H$ are free with suitable exponents.

The exponents of $\A$ are $\expA{\A}{1,7,9,13,15}$.

Now suppose that there is a hyperplane $H$ such that $\tilde{\A}$ is free.
Looking at the intersection lattice $L:= \LA{\A}$ we find the follwing multiset of invariants:
\begin{equation*}
\{\{ \Betrag{\A_X} \mid X \in L_2 \}\} = \{\{ 2^{270}, 3^{240} \}\}.
\end{equation*}
With Lemma \ref{lem:A_u_H} and the same argument 
as in the proof of Lemma \ref{lem:Add_H_A31} for $H$ we must have:
\begin{equation*}
\Betrag{ P_H } = \Betrag{\{ X \in L_2 \mid X \subseteq H \} } \geq 4.
\end{equation*}

If we look at all the possible cases for an $H$ such that $\Betrag{ P_H } \geq 2$ (there are only 2 possible cases up to symmetry)
we already see that in none of these cases the characteristic polynomial of $\tilde{\A}$ splits into linear factors over
$\PolyRing{\Ganz}{x}$ and by Theorem \ref{thm:A_free_factZ} $\tilde{\A}$ is not free.

Hence we cannot add a single hyperplane $H$ to $\A$ and obtain a free arrangement $\A \dot{\cup} \{H\} =:\tilde{\A}$
and $\A$ is not recursively free. 
\end{proof}

\subsection{The reflection arrangement $\Arr{\crg{34}}$}

In this part we will see, that the reflection arrangement $\A(W)$ with $W$ isomorphic to the finite complex reflection group
$\crg{34}$ is free but not recursively free.

\begin{lemma}\label{lem:A34_nRF}
Let $\A = \A(W)$ be the reflection arrangement with $W \cong \crg{34}$.
Then $\A$ is not recursively free.
\end{lemma}
\begin{proof}
To prove the lemma, we could follow the same path as in the proof of Lemma \ref{lem:A33_nRF}.

But since the arrangement of $\Arr{\crg{33}}$ is a parabolic subarrangement (localization) $\A_X$ 
of the reflection arrangement $\A = \Arr{\crg{34}}$ for a suitable $X \in \LA{\A}$ 
(see e.g. \cite[Table~C.15.]{orlik1992arrangements} or \cite[Ch.~7,~6.1]{lehrer2009unitary}), 
and this subarrangement is not recursively free by Lemma \ref{lem:A33_nRF},
$\A$ cannot be recursively free by Proposition \ref{prop:Arf_AXrf}.
\end{proof}

This completes the proof of Theorem \ref{thm:RFcrArr}.

\section{Abe's conjecture}\label{sec:Abes_Conjecture}

In this section we give the proof of Theorem \ref{thm:ConjAbe}, which settles \cite[Conj.~5.11]{AbeDivFree2015}.

The following result by Abe gives the divisional freeness of  the reflection arrangement $\A(\crg{31})$.
\begin{theorem}[{\cite[Cor.~4.7]{AbeDivFree2015}}]\label{thm:AbeDFRA}
Let $W$ be a finite irreducible complex reflection group and $\A = \A(W)$ its corresponding reflection arrangement.
Then $\A \in \IFC$ or $W = \crg{31}$ if and only if $\A \in \DFC$.
\end{theorem}

With results from the previous section we can now state the proof of the theorem.

\begin{proof}[{Proof of Theorem \ref{thm:ConjAbe}}]
Let $\A = \A(\crg{31})$ be the reflection arrangement of the finite complex reflection group $\crg{31}$.
Then on the one hand by Theorem \ref{thm:AbeDFRA} we have $\A \in \DFC$, but on the other hand
by Theorem \ref{thm:G31_G29_nRF} we have $\A \notin \RFC$.
\end{proof}

\begin{remark}
Furthermore, with Corollary \ref{coro:A31_FFSA_nRF}, we see that every \FFSA $\tilde{\A} \subseteq \A(\crg{31})$
still containing a hyperplane $H \in \tilde{\A}$ such that $\Betrag{\tilde{\A}^H}=31$ is in $\DFC$. 
\end{remark}

\section{Restrictions}

In \cite{MR3250448} Amend, Hoge and R\"ohrle showed, which restrictions of reflection arrangements are inductively free.
Despite the free but not inductively free reflection arrangements them self investigated in this paper, 
by \cite[Thm.~1.2]{MR3250448} there are four restrictions of reflection arrangements which remain to be inspected, namely
\begin{enumerate}
\item the $4$-dimensional restriction $(\A(\crg{33}),A_1)$, 
\item the $5$-dimensional restriction $(\A(\crg{34}),A_1)$,
\item the $4$-dimensional restriction $(\A(\crg{34}),A_1^2)$,
and
\item the $4$-dimensional restriction $(\A(\crg{34}),A_2)$,
\end{enumerate}
which are free but not inductively free (compare with \cite[App.~C.16, C.17]{orlik1992arrangements}).

Using similar techniques as for the reflection arrangements $\A(\crg{31})$, and $\A(\crg{33})$,
we can say the following about the remaining cases:
\begin{proposition}\label{prop:Res_RF}~
\begin{enumerate}
\item $(\A(\crg{33},A_1)$ is recursively free.
\item $(\A(\crg{34},A_1)$ is not recursively free.
\item $(\A(\crg{34}),A_1^2)$ is not recursively free,
and
\item $(\A(\crg{34}),A_2)$ is not recursively free.
\end{enumerate}
\end{proposition}
\begin{proof}
Let $\A$ be as in case (1). The arrangement may be defined by the following linear forms:
\begin{align*}
\A = \{ &(               1,               0,               0,               0 )^\perp,
	(               1,               1,               0,               0 )^\perp,
         (               1,               1,               1,               0 )^\perp,
        (               1,               1,               1,               1 )^\perp, 
       (               0,               1,               0,               0 )^\perp, \\
        &(               0,               1,               1,               0 )^\perp,
        (               0,               1,               1,               1 )^\perp,
        (               0,               0,               1,               0 )^\perp, 
        (               0,               0,               1,               1 )^\perp,
        (               0,               0,               0,               1 )^\perp, \\
       & (         \zeta^2,               0,               -1,         \zeta^2 )^\perp,
        (              1,               0,               -1,         \zeta^2 )^\perp,
        (          2\zeta,   2\zeta+\zeta^2,            \zeta,         -\zeta^2 )^\perp, \\
        &(              -1,   \zeta+2\zeta^2,          \zeta^2,              -1 )^\perp,
        (           \zeta,               0,               -1,         \zeta^2 )^\perp, 
        (               2,  -2\zeta-\zeta^2,               1,         -\zeta^2 )^\perp, \\
      &(           \zeta,    \zeta-\zeta^2,         2\zeta,           \zeta )^\perp,
        (         \zeta^2,  \zeta-2\zeta^2,               -1,         \zeta^2 )^\perp, \\
       &(         \zeta^2,     -\zeta+\zeta^2,       2\zeta^2,         \zeta^2 )^\perp, 
       (         \zeta^2,               0,           - \zeta,         \zeta^2 )^\perp, 
       (         \zeta^2,               0,          -\zeta^2,              1 )^\perp, \\
       &(         \zeta^2,               0,               -1,           \zeta )^\perp, 
       (          2\zeta,     \zeta-\zeta^2,       -2\zeta^2,         -\zeta^2 )^\perp,
       (           \zeta,  2\zeta+\zeta^2,               -1,         \zeta^2 )^\perp, \\
       &(        -2\zeta^2,    \zeta-\zeta^2,         2\zeta,           \zeta )^\perp, 
       (              -1,   2\zeta+\zeta^2,            \zeta,         -\zeta^2 )^\perp, \\
       &(          2\zeta,     \zeta-\zeta^2,            \zeta,         -\zeta^2 )^\perp, 
       (          2\zeta,   2\zeta+\zeta^2,            \zeta,              -1 )^\perp 
 \} \\
 = \{ &H_1,\ldots, H_{28} \},
\end{align*}
where $\zeta = \frac{1}{2}(-1+i\sqrt{3})$ is a primitive cube root of unity.

We can successively remove $6$ hyperplanes
\begin{align*}
%5, 6, 7, 13, 25, 28
\{H_5, H_6, H_7, H_{13}, H_{25}, H_{28} \} =: \{K_1,\ldots,K_6\} =: \N,
\end{align*}
with respect to this order such that $\A \setminus \N = \tilde{\A}$ is a \FFSA with a free filtration 
$\A = \A_0 \supsetneq \A_1 \supsetneq \cdots \supsetneq \A_6 = \tilde{\A}$, $\A_i = \A \setminus \{K_1,\ldots,K_i\}$.
Moreover, all the restrictions $\A_{i-1}^{K_{i}}$, ($1 \leq i \leq 6$), are inductively free.
Then we can add $2$ new hyperplanes 
\begin{align*}
\{I_1,I_2\} := \{ ( -2\zeta-3\zeta^2, 3, 2, 1 )^\perp, ( \zeta, 0, 2, 1 )^\perp \},
\end{align*}
such that $\tilde{\A_j} := \tilde{\A} \cup \{I_1,\ldots,I_j\}$, ($j=1,2$) 
are all free and
in particular the arrangement $\tilde{\A}_2 = \tilde{\A} \cup \{I_1,I_2\}$ is inductively free.
Furthermore the $\tilde{\A}_j^{I_j}$ are inductively free.
Hence $\A$ is recursively free.

The arrangement in (2) is isolated which can be seen similarly as for the arrangement $\A(\crg{33})$.

To show that the restrictions $(\A(\crg{34}),A_1^2), (\A(\crg{34},A_2)$ from (3) and (4)
are not recursively free, we look at the exponents of their 
minimal possible \FFSAs computed by Amend, Hoge, and R\"ohrle in \cite[Lemma~4.2,~Tab.~11,12]{MR3250448} and then use 
Lemma  \ref{lem:A_u_H}
and a similar argument as in the proof of Lemma \ref{lem:Add_H_FFSA}.

Let $\A$ be as in (3). Then Amend, Hoge, and R\"ohrle showed that the multiset of exponents of a 
minimal possible \FFSA $\tilde{\A} \subseteq \A$ are $\expA{\tilde{\A}}{1,13,15,15}$, (see \cite[Tab.~11]{MR3250448}).
Now, as in the proof of Lemma \ref{lem:Add_H_FFSA}, suppose $\tilde{\A} \subseteq \A$ is a \FFSAp, 
and there is a hyperplane $H$, such that $\tilde{\A} \cup \{ H \}$
is free.
Then by Lemma \ref{lem:A_u_H} we have $\sum_{X \in P_H} (\Betrag{\tilde{\A}_X} - 1) \geq 13$, where 
$P_H= \{ X \in \LAq{\tilde{\A}}{2} \mid X \subseteq H \}$. Now $\LAq{\tilde{\A}}{2} \subseteq \LAq{\A}{2}$ and
we have the following multiset if invaraints of $\A$:
\begin{align*}
\{\{ \Betrag{\A_X} \mid X \in \LAq{\A}{2} \}\} = \{\{ 2^{264},3^{304},4^{34},5^{16}\}\}.
\end{align*}
So in particular we should have $\sum_{X \in P_H} (\Betrag{\A_X} - 1) \geq 13$, and $\Betrag{P_H} \geq 4$.
If $\Betrag{P_H} = 4$ then there are at least two $X,Y \in P_H$ with $\Betrag{\A_X}=\Betrag{\A_Y} = 5$
But for all such $X,Y$ we either have $X + Y = V$ or $X+Y \in \A$.
So there is at most one $X \in P_H$ such that $\Betrag{\A_X}=5$.
If $\Betrag{P_H} = 4$ we must have at least $X,Y \in P_H$
with $\Betrag{\A_X}=5$, $\Betrag{\A_Y} = 4$.
But again for all such $X,Y$ we either have $X + Y = V$ or $X+Y \in \A$.
Considering the other cases (giving a number partition of the smalles exponent not equal to $1$)
similarly we get that $H \in \A$. 
Hence $\A$ is not recursively free.

Finally let $\A$ be as in (4).
Then Amend, Hoge, and R\"ohrle showed that the multiset of exponents of a 
minimal possible \FFSA $\tilde{\A} \subseteq \A$ are $\expA{\tilde{\A}}{1,9,10,11}$ or $\expA{\tilde{\A}}{1,10,10,10}$,
(see \cite[Tab.~12]{MR3250448}).
Suppose $\tilde{\A} \subseteq \A$ is a \FFSAp, and there is a hyperplane $H$, such that $\tilde{\A} \cup \{ H \}$
is free.
Then inspecting the intersection lattice of $\A$ analogously to case (3) we again get $H \in \A$. 
Hence $\A$ is not recursively free.
\end{proof}

Since the restrictions $(\A(\crg{34}),A_1^2)$ and $(\A(\crg{34}),A_2)$ behave somehow similar to the reflection arrangement
$\A(\crg{31})$, they also give examples for divisionally free but not recursively free arrangement, 
(compare with Theorem \ref{thm:ConjAbe} and Section \ref{sec:Abes_Conjecture}).
For further details on divisional freeness of restrictions of reflection arrangements see the recent note by R\"ohrle, \cite{Roe15}.

\renewcommand{\refname}{References}

\bibliographystyle{amsalpha}

\begin{thebibliography}{ACKN14}

\bibitem[Abe15]{AbeDivFree2015}
Takuro Abe, \emph{Divisionally free arrangements of hyperplanes}, Inventiones
  mathematicae (2015),
  \href{http://dx.doi.org/10.1007/s00222-015-0615-7}{DOI
  10.1007/s00222--015--0615--7}.

\bibitem[ACKN14]{2014arXiv1411.3351A}
T.~{Abe}, M.~{Cuntz}, H.~{Kawanoue}, and T.~{Nozawa}, \emph{{Non-recursive
  freeness and non-rigidity of plane arrangements}},
  \href{http://arxiv.org/abs/1411.3351}{arXiv:1411.3351} (2014).

\bibitem[AT15]{Abe2015}
Takuro Abe and Hiroaki Terao, \emph{Free filtrations of affine weyl
  arrangements and the ideal-shi arrangements}, Journal of Algebraic
  Combinatorics (2015),
  \href{{http://dx.doi.org/10.1007/s10801-015-0624-z}}{DOI
  10.1007/s10801--015--0624--z}.

\bibitem[AHR14]{MR3250448}
Nils Amend, Torsten Hoge, and Gerhard R{\"o}hrle, \emph{On inductively free
  restrictions of reflection arrangements}, J. Algebra \textbf{418} (2014),
  197--212.

\bibitem[BC12]{MR2854188}
Mohamed Barakat and Michael Cuntz, \emph{Coxeter and crystallographic
  arrangements are inductively free}, Adv. Math. \textbf{229} (2012), no.~1,
  691--709.

\bibitem[CH15]{MR3272729}
M.~Cuntz and T.~Hoge, \emph{Free but not recursively free arrangements}, Proc.
  Amer. Math. Soc. \textbf{143} (2015), no.~1, 35--40.

\bibitem[GAP14]{GAP4}
The GAP~Group, \emph{{GAP -- Groups, Algorithms, and Programming, Version
  4.7.6}}, 2014.

\bibitem[HR15]{2012arXiv1208.3131H}
T.~{Hoge} and G.~{R{\"o}hrle}, \emph{{On inductively free reflection
  arrangements}}, J. Reine Angew. Math \textbf{701} (2015), 205--220.

\bibitem[HRS15]{HRS15}
T.~Hoge, G.~R{\"o}hrle, and A.~Schauenburg, \emph{{Inductive and recursive
  freeness of Localizations of Multiarrangements}},
  \href{http://arxiv.org/abs/1501.06312}{arXiv:1501.06312} (2015).

\bibitem[LT09]{lehrer2009unitary}
G.~I. Lehrer and D.~E. Taylor, \emph{{Unitary Reflection Groups}}, Australian
  Mathematical Society Lecture Series, Cambridge University Press, 2009.

\bibitem[OT92]{orlik1992arrangements}
P.~Orlik and H.~Terao, \emph{Arrangements of hyperplanes}, Grundlehren der
  mathematischen Wissenschaften : a series of comprehensive studies in
  mathematics, Springer, 1992.
  
\bibitem[R{\"o}h15]{Roe15}
Gerhard R{\"o}hrle, \emph{{Divisionally free Restrictions of Reflection Arrangements}},
  \href{http://arxiv.org/abs/1510.00213}{arXiv:1510.00213} (2015).

\bibitem[ST54]{ST_1954_fcrg}
G.~C. Shephard and J.~A. Todd, \emph{Finite unitary reflection groups},
  Canadian Journal of Mathematics \textbf{6} (1954), 274--304.

\bibitem[Ter80]{Terao1980}
H.~Terao, \emph{Arrangements of hyperplanes and their freeness {I}}, J. Fac.
  Sci. Univ. Tokyo \textbf{27} (1980), 293--320.

\bibitem[Zie87]{Zielger87}
G{\"u}nter~M. Ziegler, \emph{{Algebraic Combinatorics of Hyperplane
  Arrangements}}, Ph.D. thesis, {Massachusetts Institute of Technology}, 1987.

\end{thebibliography}
\providecommand{\bysame}{\leavevmode\hbox to3em{\hrulefill}\thinspace}
\providecommand{\MR}{\relax\ifhmode\unskip\space\fi MR }
\providecommand{\MRhref}[2]{%
  \href{http://www.ams.org/mathscinet-getitem?mr=#1}{#2}
}
\providecommand{\href}[2]{#2}

\end{document}